\documentclass{amsart}
\usepackage{amsmath, amsthm, amsfonts, amssymb, enumerate}
\usepackage{mathrsfs}
\usepackage[top=3.53cm, bottom=3cm, left=2.9cm, right=2.9cm]{geometry}

\usepackage[dvipsnames]{xcolor} 
\definecolor{greenBleu}{rgb}{0.0, 0.6, 0.5}

\usepackage{hyperref}
\hypersetup{colorlinks=true, linkcolor=red, citecolor=greenBleu, urlcolor=cyan}

\newtheorem{proposition}{Proposition}[section]
\newtheorem{theorem}{Theorem}[section]
\newtheorem{lemma}[proposition]{Lemma}
\newtheorem{remark}{Remark}[section]
\newtheorem{corollary}[theorem]{Corollary}

\numberwithin{equation}{section}

\newcommand{\R}{\mathbb{R}} 
 
\newcommand{\N}{\mathbb{N}} 
\newcommand{\Z}{\mathbb{Z}} 

\newcommand{\Td}{T_{\delta}}

\newcommand{\eps}{\varepsilon}

\newcommand{\pv}{\textrm{p.v.}}

\let\Re=\undefined\DeclareMathOperator{\Re}{Re}
\let\Im=\undefined\DeclareMathOperator{\Im}{Im}

\allowdisplaybreaks

\begin{document}

\title[GWP for intermediate NLS with nonvanishing conditions at infinity]{Global well-posedness for intermediate NLS with nonvanishing conditions at infinity}

\author{Takafumi Akahori}
\address{Faculty of Engineering, 
 Shizuoka University, Jyohoku 3-5-1, Hamamatsu-shi, Shizuoka, 432-8561, Japan.}
\email{akahori.takafumi@shizuoka.ac.jp}

\author{Rana Badreddine}
\address{Department of Mathematics, University of California, Los Angeles, CA 90095, USA.}
\email{badreddine@math.ucla.edu}

\author{Slim Ibrahim}
\address{Department of Mathematics and Statistics, 
 University of Victoria, 3800 Finnerty Road, Victoria, BC V8P 5C2, Canada.}
\email{ibrahims@uvic.ca}

\author{Nobu Kishimoto}
\address{Research Institute for Mathematical Sciences, 
 Kyoto University, Oiwake-Cho, Kitashirakawa, Sakyo-ku, Kyoto-shi, Kyoto, 606-8502, Japan.}
\email{nobu@kurims.kyoto-u.ac.jp}

\begin{abstract}
The intermediate nonlinear Schr\"odinger equation (INLS) describes the dynamics of the envelope of weakly nonlinear internal waves in a stratified fluid of finite depth. While the INLS equation is known to admit dark soliton solutions, these solutions possess nonvanishing boundary conditions at spatial infinity and therefore fall outside the scope of existing well-posedness frameworks. This paper establishes the local and global well-posedness of a generalized INLS equation in Zhidkov-type spaces tailored to these nonvanishing boundary conditions. Furthermore, we rigorously justify the deep-water limit, proving that solutions of the generalized INLS converge to those of the generalized Calogero-Moser (CM) derivative NLS equation in Zhidkov-type spaces.

Our well-posedness theory relies on the modified energy method combined with frequency envelopes, marking the first application of these techniques to Zhidkov-type spaces.

\end{abstract}

\date{May 24, 2026}
\maketitle
\tableofcontents


\section{Introduction}
\label{sec:intro}

We study the intermediate nonlinear Schr\"odinger equation in one space dimension, 
\begin{equation*}\tag{INLS}\label{INLS}
i\partial_tu - \partial_x^2 u = u (i - \Td)\partial_x |u|^2 
\qquad 
\mbox{in}~\mathbb{R}\times \mathbb{R}, 
\end{equation*}
where $u$ is a complex-valued function, $\delta>0$ is a parameter called the total fluid depth, and $\Td$ is a singular integral operator defined by
\[ \Td f (x) := \frac{1}{2\delta}\lim_{\varepsilon \downarrow 0} \int_{|x-y|>\varepsilon} \coth{\Big( \frac{\pi (x-y)}{2\delta}\Big)}f(y)\,dy . \]
The equation \eqref{INLS} was introduced by Pelinovsky in \cite{P} to describe the long-term dynamics of the envelope of weakly nonlinear internal waves in a stratified fluid of finite depth.

A notable intrinsic property of \eqref{INLS} is that it admits \emph{dark solitons} (see \cite{P,PG,Matsuno1}). 
Specifically, for any background wave density $\rho>0$ and integer $N\geq 1$, there exists an $N$-soliton solution $Q$ to \eqref{INLS} satisfying the nonvanishing boundary condition $|Q(x,t)|\to \rho$ as $|x|\to \infty$ for all $t\in\R$, while its spatial derivative $\partial_x Q$ remains in $L^2(\mathbb{R})$. 
Although the existence and asymptotic properties of these dark solitons are well-documented (see \cite{Akahori}), determining local and global well-posedness in a functional setting appropriate for such solutions has remained an open problem. 
To date, the well-posedness literature for \eqref{INLS} has been restricted to spaces of decaying functions. 
For instance, local well-posedness was established in $H^s(\R)$ for $s\geq1$ by de\,Moura~\cite{DeMoura}, and later for small initial data in the Besov space $B_2^{1/2,1}(\R)$ by Barros--de\,Moura--Santos \cite{B-deM-S}. 
More recently, Chapouto--Forlano--Laurens \cite{CFL} showed that \eqref{INLS} is globally well-posed in $H^s(\R)$ for $s>1/4$ under a smallness assumption on the $L^2$ initial data. 
However, none of these functional frameworks can accommodate the nonvanishing boundary conditions intrinsic to dark solitons. 
The primary objective of this paper is to bridge this gap by establishing the global well-posedness of \eqref{INLS} in a Zhidkov-type space (see Theorems \ref{thm:LWP} and \ref{thm:GWP} below).

To properly address these nonvanishing boundary conditions, we first recall the definition of the Banach Zhidkov space: for any $k\in \N,$
\[ X^k:=\{ u\in L^\infty(\R):\partial_xu\in H^{k-1}(\R)\},\qquad \| u\|_{X^k}:=\| \partial_xu\|_{H^{k-1}}+\| u\|_{L^\infty}. \]
We then introduce the following spaces defined, for any $k\in \N$, $\rho\geq 0$ as
\begin{equation}\label{X}
\begin{gathered}
\mathcal{Z}^k_\rho:=\{ u\in L^{\infty}(\R) : \partial_xu\in H^{k-1}(\R)\,,|u|^2-\rho^2 \in L^2(\R) \}, \\
\mathcal{Z}^k:=\bigcup_{\rho\geq 0}\mathcal{Z}^k_\rho =\{ u\in L^\infty(\R) : \partial_xu\in H^{k-1}(\R)\,,|u|^2-\rho^2\in L^2(\R)~\text{for some $\rho\geq 0$}\} .
\end{gathered}
\end{equation}
These function spaces naturally accommodate dark solitons. 
Note that $\mathcal{Z}^k$ is not a vector space, but rather a complete metric space equipped with the distance function
\begin{equation}\label{def:dk}
d^k(f,g):=\| f-g\|_{X^k}+\big\| (|f|^2-\rho_f^2)-(|g|^2-\rho_g^2)\big\|_{L^2},
\end{equation}
where, for any $f\in \mathcal{Z}^k$, the constant $\rho_f$ denotes the unique non-negative number such that 
\[ |f|^2-\rho_f^2\in L^2(\R). \]
We also note that for each $\rho\geq 0$, the space $\mathcal{Z}^k_\rho$ is a closed subset of $\mathcal{Z}^k$, and that $H^k(\R) \hookrightarrow \mathcal{Z}^k_0 \hookrightarrow X^k$.
To measure the size of functions in $\mathcal{Z}^k$, we introduce the functionals
\begin{equation}\label{def:Ek}
\begin{aligned} 
\widetilde{\mathcal{E}} (f) &:=\big\| |f|^2-\rho_f^2 \big\|_{L^2} , \\ 
\mathcal{E}^k(f)&:= \| f\|_{X^k}^2+\widetilde{\mathcal{E}} (f), \\
\mathcal{E}^k_\delta (f) &:=\| f \|_{X^k}^2 + \tfrac{1}{\delta^{4/3}} \widetilde{\mathcal{E}}(f) \qquad (\delta>0).
\end{aligned} 
\end{equation}

While \eqref{INLS} is famously known as a completely integrable system \cite{PG,CFL} the analytical framework developed in this paper does not rely on these integrable structures. 
To underscore the robustness of our well-posedness theory, we extend our analysis to a broader class of models by introducing the following generalized equation
\begin{equation}\label{gINLS}\tag{gINLS}
i\partial_tu-\partial_x^2u=u(i\alpha -\beta T_\delta)\partial_x|u|^2 \qquad \text{in $\mathbb{R}\times \mathbb{R}$,}
\end{equation}
where $\alpha,\beta\in \R$ are constants. 
Note that \eqref{gINLS} recovers its completely integrable structure when $\beta =\pm |\alpha|\neq 0$, with the plus (resp. minus) sign corresponding to the defocusing (resp. focusing) nonlinearity%
\footnote{~The sign of $\alpha$ is not relevant. Indeed, if $\alpha \neq 0$ one can set $\alpha =1$ by the change of variables $u(x,t)\mapsto |\alpha |^{1/2}u((\text{sgn}\,{\alpha})x,t)$, by which the nonlinearity is transformed into $u(i-\beta |\alpha|^{-1}\Td)\partial_x|u|^2$.}%
. 
The original \eqref{INLS} is naturally recovered by setting $\alpha =\beta =1$.

\medskip
We are now ready to state our first main result, which establishes local well-posedness for \eqref{gINLS} in $\mathcal{Z}^2$.

\begin{theorem}[Local well-posedness in $\mathcal{Z}^2$]\label{thm:LWP}
Let $\alpha, \beta\in \R$, $\delta>0$ and $M_0>0$. 
For any initial data $\phi\in \mathcal{Z}^2$ satisfying $\mathcal{E}^2(\phi)\leq M_0^2$, there exists a time $T_*\in (0,1]$, bounded below by
\[ T_*\geq c_{\delta,\alpha,\beta}(1+M_0^4)^{-1} \]
(for some constant $c_{\delta,\alpha,\beta}>0$), such that \eqref{gINLS} admits a unique solution $u\in C([0,T_*];\mathcal{Z}^2)$ with $u(0)=\phi$. 
Furthermore, the solution satisfies the following properties:
\begin{enumerate}
\item[{\upshape (i)}] The mapping $\phi\mapsto u$ is continuous from the metric ball $\{ \phi\in \mathcal{Z}^2\colon\mathcal{E}^2(\phi)\leq M_0^2\}$ into $C([0,T_*];\mathcal{Z}^2)$.
\item[{\upshape (ii)}] The background wave density is preserved, such that $u\in C([0,T_*];\mathcal{Z}^2_{\rho_\phi})$.
\item[{\upshape (iii)}] If $\phi\in \mathcal{Z}^k$ for some integer $k\geq 3$, the regularity persists: $u\in C([0,T_*];\mathcal{Z}^k)$.
\end{enumerate}
Finally, if $\delta\geq 1$, the time of existence admits a uniform lower bound $T_*\geq c_{\alpha,\beta}(1+M_0^4)^{-1}$ independent of $\delta$, and the same well-posedness results hold for initial data satisfying the weaker bound $\mathcal{E}^2_\delta(\phi)\leq M_0^2$.
\end{theorem}

\begin{remark}\hfill
\begin{enumerate}
\item[{\upshape (1)}] 
The final statement of the theorem utilizes the weaker bound $\mathcal{E}^2_\delta(\phi) \leq M_0^2$ to show that for large depths ($\delta \geq 1$), the uniform time of existence $T_*$ survives even if the non-Zhidkov energy grows as large as $\widetilde{\mathcal{E}}(\phi) = O(\delta^{4/3})$. 
This naturally bridges the $\mathcal{Z}^2$ local theory to the deep-water limit, which is well-posed in the broader space $X^2$ where the non-Zhidkov part is entirely unconstrained (see Corollary~\ref{cor:LWP-CM}). 
The specific exponent $-4/3$ arises from the $L^\infty$ estimates on the nonlinearity \eqref{est:aprioriZ2-A}, though its optimality remains an open question.
\item[{\upshape (2)}] 
While \cite{Chen} claimed well-posedness for the deep-water limit using Kato's iterative scheme, this approach is fundamentally problematic for \eqref{INLS} or \eqref{gINLS} in spaces with nonvanishing boundaries. 
Kato's scheme requires expanding the nonlinearity, but individual terms like $T_\delta(\bar{u}\partial_x u)$ suffer from a severe low-frequency divergence because $\int \bar{u}\partial_x u \,dx \neq 0$ for non-decaying functions. 
We avoid this infrared blowup by employing a modified energy method \cite{shatah,colliander-kdv}, which robustly treats $\partial_x(|u|^2)$ as an unbroken exact derivative.
\end{enumerate}
\end{remark}

The uniform lower bound on the time of existence $T_*$ for large fluid depths ($\delta \geq 1$) naturally invites a rigorous investigation into the deep-water limit. 
As $\delta \to \infty$, the singular integral operator $T_\delta$ formally converges to the Hilbert transform $H$, defined by
\[ H f(x) := \lim_{\varepsilon \downarrow 0} \int_{|x-y|>\varepsilon} \frac{f(y)}{\pi (x-y)}\,dy ,\qquad \widehat{H f}(\xi)=-i\,\mathrm{sgn}(\xi) \hat{f}(\xi). \]
In this limit, the \eqref{gINLS} equation transitions into the generalized Calogero-Moser derivative nonlinear Schr\"odinger equation:
\begin{equation}\label{gCM}\tag{gCM}
i\partial_tu-\partial_x^2u=u(i\alpha -\beta H)\partial_x|u|^2\qquad \text{in $\mathbb{R}\times \mathbb{R}$}.
\end{equation}
Because the time of existence in Theorem~\ref{thm:LWP} is independent of $\delta$ for $\delta \ge 1$, our local theory seamlessly extends to the limiting case $\delta=\infty$. 
This yields local well-posedness for \eqref{gCM} in the broader space $X^2$, where the non-Zhidkov energy $\widetilde{\mathcal{E}}(\cdot)$ is unrestricted.

\begin{corollary}[Local well-posedness for \eqref{gCM}]\label{cor:LWP-CM}
Let $\alpha, \beta \in \mathbb{R}$ and $M_0>0$. 
For any initial data $\phi\in X^2$ satisfying $\| \phi\|_{X^2}\leq M_0$, there exists a time $T_*\in (0,1]$ satisfying $T_*\geq c_{\alpha,\beta}(1+M_0^4)^{-1}$ such that \eqref{gCM} admits a unique solution $u\in C([0,T_*];X^2)$ with $u(0)=\phi$.
The solution depends continuously on the initial data and persists higher regularities.
Furthermore, if $\phi\in \mathcal{Z}^2$, the solution $u$ remains in $C([0,T_*];\mathcal{Z}^2_{\rho_\phi})$ and the mapping $\phi\mapsto u$ is continuous in the $\mathcal{Z}^2$ topology.
\end{corollary}

Beyond the independent well-posedness of both systems, the following result establishes the rigorous convergence of \eqref{gINLS} to \eqref{gCM} in Zhidkov-type spaces as the depth parameter  $\delta\to \infty$. 

\begin{theorem}[Deep-water limit convergence]\label{thm:conv}
Let $\alpha, \beta \in \mathbb{R}$. 
Consider initial data $\phi_\infty\in X^2$ and a family $\{ \phi _\delta\}_{\delta \geq 1} \subset \mathcal{Z}^2$ such that 
\[ \lim_{\delta\to \infty} \| \phi_\delta -\phi_\infty\|_{X^2}=0 \quad \text{and} \quad \lim_{\delta\to \infty}\delta^{-\frac{4}{3}}\widetilde{\mathcal{E}}(\phi_\delta)=0. \]
Let $u_{\infty} \in C([0,T_{\infty,\max});X^2)$ and $u_\delta \in C([0,T_{\delta,\max});\mathcal{Z}^2)$ be the maximal solutions to \eqref{gCM} and \eqref{gINLS} with initial data $\phi_\infty$ and $\phi_\delta$, respectively. 
Then:
\begin{enumerate}
\item[{\upshape (i)}] $\liminf_{\delta\to\infty}T_{\delta,\max}\geq T_{\infty,\max}$.
\item[{\upshape (ii)}] For any $T \in (0, T_{\infty,\max})$, $u_\delta$ converges uniformly to $u_\infty$ in $X^2$ on $[0,T]$.
\item[{\upshape (iii)}]  If additionally $\phi_\infty\in \mathcal{Z}^2$ and $\lim_{\delta\to \infty}d^2(\phi_\delta,\phi_\infty)=0$, this uniform convergence holds in the $\mathcal{Z}^2$ metric: for any $T \in (0, T_{\infty,\max})$,
\[ \lim_{\delta\to \infty}\max_{t\in [0,T]}d^2\big( u_\delta(t),u_\infty(t)\big)=0. \]
\end{enumerate}
\end{theorem}

\medskip
Having established local well-posedness and asymptotic convergence, we turn to the question of \emph{global existence}. 
While the original \eqref{INLS} model possesses a completely integrable structure with an infinite hierarchy of conservation laws derived via Lax pairs (see, e.g., \cite{PG,CFL}), this standard integrable machinery breaks down in our setting. 
Indeed, writing the first conservation laws of \eqref{INLS} using the formula
\begin{equation}\label{def:lax}
\int _{\mathbb{R}} \overline{u}\mathcal{L}_u^nu\,dx\quad (n=0,1,2,\dots),\qquad \mathcal{L}_u:\ f\in H^1(\R)\mapsto -i\partial_xf+u\frac{1+i\Td}{2}(\overline{u}f)\in L^2(\R),
\end{equation}
one obtains,
\begin{align*}
E_0(u)&=\int_{\mathbb{R}} |u|^2\,dx ,\\
E_{1/2}(u)&=\int_{\mathbb{R}} \Big\{ \Im [\overline{u}\partial_xu]+\frac12|u|^4\Big\} \,dx ,\\
E_1(u)&=\int_{\mathbb{R}} \Big\{ |\partial_xu|^2+|u|^2\Im[\overline{u}\partial_xu]+\frac12|u|^2\Td\partial_x(|u|^2)+\frac13|u|^6\Big\} \,dx ,\\
&\vdots
\end{align*}

However, these integrals \emph{fail to be well-defined for} $u\in \mathcal{Z}^2$ with $\rho_u>0$, even after substituting $|u|^{2(n+1)}-\rho_u^{2(n+1)}$ for the last term $|u|^{2(n+1)}$ in $E_{n/2}(u)$. 
To overcome this fundamental roadblock, we abandon the Lax pair formalism. 
Instead, we employ a modified energy method to construct two novel, rigorously defined conservation laws for \eqref{gINLS} that are highly compatible with the nonvanishing boundaries of $\mathcal{Z}^2$. 
These invariants allow us to control the $X^2$ norm and the non-Zhidkov energy globally in time.

\begin{theorem}[Global well-posedness and a priori bounds]\label{thm:GWP}
Let $\alpha\neq 0$ and $\beta \geq 0$ be arbitrary.  
Then, the Cauchy problems for both \eqref{gINLS} with an arbitrary $\delta>0$ and \eqref{gCM} are globally well-posed in 
\[ \mathcal{Z}^2_*:=\bigcup_{\rho>0}\mathcal{Z}^2_\rho .\]
Furthermore, the following uniform estimate holds in the integrable-defocusing case $\beta=|\alpha|\neq 0$: 
\[ \sup_{t \in \mathbb{R}} \mathcal{E}^2\big( u(t)\big) <\infty. \]
\end{theorem}

\begin{remark}
In the general defocusing case ($\alpha\neq 0$ and $\beta \geq 0$), we still obtain a uniform-in-time bound for $\mathcal{E}^1(u(t))$ but only an exponential bound for $\| \partial_x^2u(t)\|_{L^2}$ (see Theorem~\ref{thm:conservation} in Section~\ref{sec4}). 
The global theory is built on two new conservation laws compatible with nonvanishing boundary conditions, which allow us to control $\|u \|_{L^{\infty}}^2$, $\| \partial_xu \|_{H^1}^2$, and $\| |u|^2-\rho_u^2\|_{L^2}$ globally in time.
\end{remark}

\medskip
\noindent\textbf{Asymptotic limits and broader context.} 
Our study of \eqref{INLS} is deeply motivated by its role as an interpolating model between two completely integrable equations. 
The singular integral operator $\Td$ is characterized by its Fourier symbol%
\footnote{~We will give a proof of the formula \eqref{T[delta]-Fourier} in Appendix~\ref{app:Fourier}.}%
:
\begin{equation}\label{T[delta]-Fourier}
\widehat{\Td f}(\xi)=-i\coth(\delta\xi)\hat{f}(\xi),
\end{equation}
where $\widehat{\cdot}$ denotes the Fourier transform. 
In the shallow-water limit, as the depth parameter $\delta\to 0$, substituting the rescaled variable $\psi (x,t)=\sqrt{\delta}u(x,t)$ formally reduces \eqref{INLS} to the standard defocusing nonlinear Schr\"odinger equation:
\begin{equation}\label{NLS}\tag{NLS}
i\partial_t \psi-\partial_x^2 \psi =-\psi |\psi |^2.
\end{equation}
Conversely, as discussed prior to Corollary~\ref{cor:LWP-CM}, the deep-water limit $\delta\to \infty$ formally yields the defocusing Calogero-Moser derivative NLS equation \eqref{gCM}.

While \eqref{INLS} itself lacks scaling invariance due to the parameter-dependent operator $\Td$, its two limiting equations possess distinct scaling symmetries. 
The \eqref{NLS} model is invariant under the $H^{-1/2}$-scaling $\psi (x,t) \mapsto \lambda \psi (\lambda x, \lambda^2 t)$, whereas the \eqref{gCM} model is invariant under the $L^2$-scaling $u(x,t) \mapsto \sqrt{\lambda} u (\lambda x, \lambda^2 t)$. 
Given the rigorous convergence of \eqref{gINLS} to the Calogero-Moser model established in Theorem~\ref{thm:conv}, it is natural to expect that \eqref{INLS} roughly inherits this $L^2$-critical behavior, making $L^2(\mathbb{R})$ its effective scaling-critical space.

The well-posedness of these limiting equations has been extensively studied. 
For the shallow-water limit, Harrop-Griffiths--Killip--Vi\c{s}an \cite{HGKV} recently proved that \eqref{NLS} is globally well-posed in $H^{s}(\mathbb{R})$ for any $s>-1/2$ (which is sharp, as it is known that \eqref{NLS} is ill-posed at the critical space $H^{-1/2}(\mathbb{R})$). 
For the deep-water limit, global well-posedness for the Calogero-Moser model was established in the Hardy space $L^2_+(\mathbb{T})$ \cite{Badreddine1} and in $L^2_+(\mathbb{R})$ \cite{killip}. 
Its focusing variant is also globally well-posed under various smallness or optimal mass assumptions \cite{Badreddine1,GL,killip,Hogan-Kowalski,kim-kim-kwon,kim-Kwon,jJeong-Kim}. 
Additional studies on the Calogero-Moser DNLS equations include inverse scattering theory \cite{Matsuno7, Frank-Read}, the zero dispersion limit \cite{Matsuno3,Badreddine3}, and the derivation of solitons \cite{GL,Matsuno2,Matsuno4,Matsuno9,Matsuno5}. 
These solitons are known to exhibit turbulent behavior on $\R$ \cite{GL} and possess a rich structure of traveling waves on $\mathbb{T}$ \cite{Badreddine2,Matsuno2,Matsuno9,Matsuno5}.

\medskip
\noindent\textbf{Organization of the paper.} 
Section~\ref{sec2} collects preliminary estimates and details the properties of the operator $T_\delta$ alongside the metric spaces $\mathcal{Z}^k$. Section~\ref{sec3} is dedicated to the proofs of local well-posedness in $\mathcal{Z}^2$ (Theorem~\ref{thm:LWP}) and the deep-water limit convergence (Theorem~\ref{thm:conv}). 
Finally, Section~\ref{sec4} establishes the modified conservation laws and constructs the global well-posedness theory (Theorem~\ref{thm:GWP}).

\medskip
{\large \textbf{Acknowledgments.}}
T.A. was supported by JSPS KAKENHI Grants Number 20K03697 and 25K07082. 
R.B. was supported by an AMS-Simons Travel Grant. 
R.B. gratefully acknowledges the University of Victoria and the Pacific Institute for the Mathematical Sciences (PIMS) for their support of the visit to UVic. 
S.I. acknowledges financial support from the Natural Sciences and Engineering Research Council of Canada (NSERC) under Grant No.\;371637-2025. 
Part of this work was carried out during an extended visit by S.I. to the Research Institute for Mathematical Sciences (RIMS). 
S.I. is grateful to colleagues and the staff at RIMS for their warm hospitality and stimulating research environment.
Part of this work was carried out during visits by R.B. and S.I. to Shizuoka University. 
They would like to thank their colleagues and the staff at Shizuoka University for their warm hospitality.
N.K. was supported by JSPS KAKENHI Grant Number 25K07066.


\section{Preliminaries}
\label{sec2}

We begin by recording the notation used throughout the paper and present several useful inequalities that will be invoked repeatedly in the analysis.

\subsection{Notation}
\label{subsec:notation}

We use $\mathcal{F}$ (or $\widehat{\,\cdot\,}$) and $\mathcal{F}^{-1}$ to denote the Fourier and its inverse transformations:
\[ \mathcal{F}[f](\xi)=\widehat{f}(\xi) :=\int_{\R}e^{-ix\xi}f(x)\,dx,\qquad \mathcal{F}^{-1}[g](x):=\frac1{2\pi}\int_{\R}e^{ix\xi}g(\xi)\,d\xi .\]
We use $U(t)$ to denote the free Schr\"odinger operator given by
\[ U(t):=e^{-it\partial_x^2}=\mathcal{F}^{-1}e^{it|\xi|^2}\mathcal{F}. \]
Our convention for the $L^2$-inner product $\left( f,g\right)_{L^2}$ is
\[ \left( f,g\right)_{L^2} :=\int_\R \overline{g}(x)  f(x) \, dx. \]

For the frequency localization, let $\varphi$ be a smooth, nonnegative radial function supported on $|\xi| \leq 2$ with $\varphi(\xi) = 1$ for $|\xi| \leq 1$. 
For each nonnegative integer $j\in \Z_{\geq 0}$, define the (inhomogeneous) Littlewood-Paley projection $P_{\leq j}$ as the Fourier multiplier with symbol $\varphi(\xi/2^j)$, and set $P_{\geq j+1} := I - P_{\leq j}$, $P_j:=P_{\leq j}-P_{\leq j-1}$ ($j\geq 1$) and $P_0:=P_{\leq 0}$.

For $T>0$, we use the convention $L_{T}^{\infty}$ to denote the space $L^{\infty}(0,T)$. 
Furthermore, for a Banach space $B$, we use $L^{\infty}_{T}B$ to denote $L^{\infty}((0,T) ;B)$.

To streamline the analysis of \eqref{gINLS}, we denote the full nonlinearity by
$$ 
    \mathcal{N}(u):=iu(\alpha +i\beta \Td)\partial_x|u|^2.
$$
By introducing the difference operator $$\mathcal{L}_\delta:=(\Td -H)\partial_x,$$ we can decompose this nonlinearity into a depth-dependent perturbation and the deep-water limit core:
\begin{equation} \label{def:N0N1}\begin{aligned}
\mathcal{N}(u)
&=-\beta u\mathcal{L}_\delta |u|^2 +iu(\alpha +i\beta H)\partial_x|u|^2 \\
&=: \mathcal{N}_\delta(u)+\mathcal{N}_\infty(u) .
\end{aligned}
\end{equation}
Recall the metric space $\mathcal{Z}^k$, the base distance $d^k(\cdot,\cdot)$, and the energy functionals $\widetilde{\mathcal{E}}(\cdot)$, $\mathcal{E}^k(\cdot)$, and $\mathcal{E}^k_\delta(\cdot)$ defined previously in \eqref{X}, \eqref{def:dk}, and \eqref{def:Ek}. For any $f,g\in \mathcal{Z}^k$, we further introduce the non-Zhidkov distance $\widetilde{d}$ and the $\delta$-scaled distance $d^k_\delta$ :
\begin{align}\label{def:d1}
\widetilde{d}(f,g)
&:=\big\| (|f|^2-\rho_f^2)-(|g|^2-\rho_g^2)\big\|_{L^2},\\
d^k_\delta(f,g)&:=\| f-g\|_{X^k}+\delta^{-\frac{4}{3}}\widetilde{d}(f,g).\notag
\end{align}
The smooth subset of these spaces is denoted by $\mathcal{Z}^\infty:=\bigcap_{k\geq 1}\mathcal{Z}^k$. Finally, we use the notation $A \lesssim B$ to indicate that there exists a generic constant $C > 0$ such that $A \leq CB$, and we write $A \sim B$ when both $A \lesssim B$ and $B \lesssim A$ hold.


\subsection{Properties of \texorpdfstring{$\Td$}{T[delta]}}

Having an unbounded symbol \eqref{T[delta]-Fourier}, the operator $\Td$ itself is not bounded on standard $L^p(\R)$ or $H^s(\R)$.
However, the difference operator $\mathcal{L}_\delta = (\Td - H)\partial_x$ is significantly better behaved. Its Fourier multiplier is given by $r_\delta(\xi) = \big(\coth(\delta\xi) - \mathrm{sgn}(\xi)\big)\xi$, which satisfies the uniform bound $|r_\delta(\xi)| \lesssim \delta^{-1}$. By Plancherel's theorem, it follows that $\mathcal{L}_\delta$ is bounded on $H^s(\R)$ for any $s \in \R$, yielding the estimate
\begin{equation}\label{est:Ldelta}
\| \mathcal{L}_\delta f\|_{H^s}\lesssim \delta^{-1}\| f\|_{H^s}.
\end{equation}
Since the Hilbert transform $H$ is a bounded isometry on $H^s(\R)$, combining \eqref{est:Ldelta} with the decomposition $\Td\partial_x = H\partial_x + \mathcal{L}_\delta$ demonstrates that $\Td\partial_x$ maps $H^{s+1}(\R)$ boundedly into $H^s(\R)$:
\begin{equation}\label{est:Tdelta}
\| \Td\partial_xf\|_{H^s}\lesssim \| \partial_xf\|_{H^s}+\delta^{-1}\| f\|_{H^s} .
\end{equation}
When working with functions $f \in \mathcal{Z}^k$, we can circumvent the lack of decay by exploiting the algebraic identity $\partial_x|f|^2 = \partial_x(|f|^2 - \rho_f^2)$. Because $|f|^2 - \rho_f^2 \in L^2(\R)$ by the definition of the Zhidkov space, applying \eqref{est:Ldelta} and \eqref{est:Tdelta} immediately yields the following crucial $L^2$ estimates:
\begin{equation}\label{est:TLonZ}
\begin{aligned}
\| \Td\partial_x|f|^2\|_{L^2}&\lesssim \delta^{-1}\big\| |f|^2-\rho_{f}^2\big\|_{L^2}+\| f\|_{L^\infty}\| \partial_xf\|_{L^2},\\
\| \mathcal{L}_\delta |f|^2\|_{L^2}&\lesssim \delta^{-1}\big\| |f|^2-\rho_{f}^2\big\|_{L^2}. \end{aligned}
\end{equation}
Finally, we observe that the Fourier symbol of $\Td\partial_x$ is $\xi \coth(\delta \xi)$. Because this symbol is strictly real, even, and non-negative, the operator $\Td\partial_x$ preserves real-valuedness ($\overline{\Td\partial_xf} = \Td\partial_x\overline{f}$), is symmetric, and is positive semi-definite. Specifically, for any suitable functions $f$ and $g$, we have:
\begin{equation}\label{id:Tdeltasymm}
\left( \Td\partial_xf,f\right)_{L^2} \geq 0,\quad \left( \Td\partial_x f,g\right)_{L^2} =\left( f,\Td\partial_xg\right)_{L^2} ,\qquad \int_{\R} f\Td\partial_xg\,dx=\int_{\R}g\Td\partial_xf\,dx ,
\end{equation}
The exact same symmetry and positivity properties hold for the deep-water limit operator $H\partial_x$, whose corresponding Fourier symbol is simply $|\xi|$.


\subsection{Properties of the space \texorpdfstring{$\mathcal{Z}^k$}{Z[k]} and control of Littlewood-Paley operators}
\label{subsec:Zk}

We begin with a fundamental $L^\infty$ bound for functions in $\mathcal{Z}^k$, which will be crucial for the global theory in Section~\ref{sec4}. For any $f\in \mathcal{Z}^1_\rho$ with $\rho\geq 0$, we have
\begin{equation}\label{est:Linfty}
\rho \leq \| f\|_{L^\infty} \lesssim \rho+\big\| |f|^2-\rho^2\big\|_{L^2}^{\frac13}\|\partial_xf\|_{L^2}^{\frac13}.
\end{equation}
The lower bound follows trivially from the fact that $|f|^{2}-\rho^{2} \in L^{2}(\mathbb{R})$. The upper bound is derived by combining the Gagliardo-Nirenberg and Young inequalities
\[ \begin{aligned}
\| f\|_{L^\infty}^2&\leq \rho^2+\big\| |f|^2-\rho^2\big\|_{L^\infty} \leq \rho^2+C\big\| |f|^2-\rho^2\big\|_{L^2}^{\frac12}\| f\|_{L^\infty}^{\frac12}\| \partial_xf\|_{L^2}^{\frac12} \\
&\leq \rho^2+\frac12\| f\|_{L^\infty}^2+C\big\| |f|^2-\rho^2\big\|_{L^2}^{\frac23}\| \partial_xf\|_{L^2}^{\frac23}. \end{aligned} \]

Next, because of the Sobolev embedding $H^1(\R)\hookrightarrow L^\infty(\R)$, the space $\mathcal{Z}^k_\rho$ is closed under the addition of standard $H^k(\R)$ functions. For any integer $k\geq 1$, $\rho\geq 0$, $f\in \mathcal{Z}^k_\rho$, and $g\in H^k(\R)$, we have $f+g\in \mathcal{Z}^k_\rho$ along with the energy bounds
\begin{equation}\label{est:Ekf+g}
\begin{aligned}
\| f+g\|_{X^k}\leq \| f\|_{X^k}+\| g\|_{X^k} \quad \text{and} \quad
\widetilde{\mathcal{E}}(f+g)\lesssim \widetilde{\mathcal{E}}(f)+\big( \| f\|_{L^\infty}+\| g\|_{L^\infty}\big) \| g\|_{L^2}.
\end{aligned}
\end{equation}
Similarly, for  $f_1, f_2 \in \mathcal{Z}^1$  and $g_1, g_2 \in H^1(\R)$, one infers
\begin{equation}\label{est:d1f+g}
\begin{aligned}
\| (f_1+g_1)-(f_2+g_2)\|_{X^1}&\leq \| f_1-f_2\|_{X^1}+\| g_1-g_2\|_{X^1} \\
\widetilde{d}(f_1+g_1,f_2+g_2)&\lesssim \widetilde{d}(f_1,f_2)+\| g_1\|_{L^2}\| f_1-f_2\|_{L^\infty}\\
&\quad +(\| f_2\|_{L^\infty}+\| g_1\|_{L^\infty}+\| g_2\|_{L^\infty})\| g_1-g_2\|_{L^2}.
\end{aligned}
\end{equation}

Finally, we collect several smoothing and convergence properties for the Littlewood-Paley projectors $P_{\leq \ell}$ acting on Zhidkov spaces. These estimates follow smoothly from standard Bernstein inequalities. For any $\rho\geq 0$, $\ell\geq 0$, and integers $k, j\geq 1$:

\begin{itemize}
    \item The operator $P_{\leq \ell}$ maps $\mathcal{Z}^k_\rho \to \mathcal{Z}^{k+j}_\rho$, satisfying:
    \begin{equation}\label{est:EkPl}
    \| P_{\leq\ell}f\|_{X^k}\leq \| f\|_{X^k},\qquad \| P_{\leq \ell}f\|_{X^{k+j}}\lesssim 2^{j\ell}\| f\|_{X^k}.
    \end{equation}
    The non-Zhidkov energy is bounded by tracking the high-frequency tails
    \[\begin{aligned}
\widetilde{\mathcal{E}}(P_{\leq \ell}f)&\leq \widetilde{\mathcal{E}}(f) +\big\| |P_{\leq \ell}f|^2-|f|^2\big\|_{L^2} \leq \widetilde{\mathcal{E}}(f)+2\| f\|_{L^\infty}\| P_{\geq \ell+1}f\|_{L^2} \\
&\leq \widetilde{\mathcal{E}}(f)+2^{1-\ell}\| f\|_{L^\infty}\| P_{\geq \ell+1}\partial_xf\|_{L^2}\leq \widetilde{\mathcal{E}}(f)+2^{1-\ell}\| f\|_{X^1}^2.
\end{aligned}\]

    \item  As $\ell \to \infty$, $P_{\leq \ell}f \to f$ in the $\mathcal{Z}^k$ topology. Specifically,
    \begin{equation}\label{est:dconv}
\begin{aligned}
\| P_{\leq \ell}f-f\|_{X^k}
&\leq \| P_{\geq \ell+1}\partial_xf\|_{H^{k-1}}+\| P_{\geq \ell+1}f\|_{L^\infty} \lesssim \| P_{\geq \ell+1}\partial_xf\|_{H^{k-1}} +2^{-\frac12\ell}\| P_{\geq \ell+1}\partial_xf\|_{L^2}, \\
\widetilde{d}(P_{\leq \ell}f,f)
&=\big\| |P_{\leq \ell}f|^2-|f|^2\big\|_{L^2} \leq 2^{1-\ell}\| f\|_{X^1}^2. 
\end{aligned}
\end{equation}
    Note that if $f\in C([0,T];\mathcal{Z}^k)$, this convergence holds uniformly in time on $[0,T]$.

    \item  To bound the distance between two projected functions, we use the algebraic expansion of $|Af|^{2}-|Ag|^{2}$ (applied with the operator $A = P_{\leq \ell}$)
    \begin{equation}\label{id:sqdiff}
    \begin{aligned}
    |Af|^{2}-|Ag|^{2} &= |f|^2-|g|^2 -(1-A)(f-g)\cdot \overline{Af} - (1-A)g\cdot \overline{A (f-g)} \\
    &\quad - (f-g) \cdot \overline{(1-A)f} - g \cdot \overline{(1-A)(f-g)}.
    \end{aligned}
    \end{equation}
    Applying this identity yields the distance bound
    \begin{equation}\label{est:dkPl} 
    \widetilde{d}(P_{\leq \ell}f,P_{\leq \ell}g)\lesssim \widetilde{d}(f,g)+2^{-\ell}\big( \| f\|_{X^1}+\| g\|_{X^1}\big) \| f-g\|_{X^1}.
    \end{equation}
    Furthermore, for different cutoffs $\ell, m \in \Z_{\geq 0}$ and functions $f,g\in \mathcal{Z}^2$, combining the triangle inequality with \eqref{est:dconv} gives:
    \begin{equation} \label{est:d1PlPm}
    \begin{aligned}
    \| P_{\leq \ell}f-P_{\leq m}g\|_{X^1}&\lesssim \| f-g\|_{X^1}+2^{-\ell}\| f\|_{X^2}+2^{-m}\| g\|_{X^2}, \\
    \widetilde{d}(P_{\leq \ell}f,P_{\leq m}g)&\lesssim \widetilde{d}(f,g)+2^{-\ell}\| f\|_{X^1}^2+2^{-m}\| g\|_{X^1}^2.
    \end{aligned}
    \end{equation}
\end{itemize}


\subsection{Estimates for the free Schr\"odinger operator}
\label{subsec:Ut}

The action of $U(t):=e^{-it\partial_x^2}$ on Zhidkov-type spaces is well understood (see, e.g., \cite{Gerard,Chen}). Applying standard Fourier analysis alongside the pointwise bound $|(e^{it\xi^2}-1)/\xi|\lesssim |t|^{1/2}$ reveals that
\begin{equation}\label{est:UtL2}
\| U(t)\phi -\phi\|_{L^2}\lesssim |t|^{\frac12}\| \partial_x\phi\|_{L^2}.
\end{equation}
Combining \eqref{est:UtL2} with the Gagliardo-Nirenberg inequality yields $L^\infty$ control 
\begin{equation}\label{est:UtLinfty}
\begin{aligned}
\| U(t)\phi-\phi\|_{L^\infty}&\lesssim |t|^{\frac14}\| \partial_x\phi\|_{L^2},\\
\| U(t)\phi\|_{L^\infty}&\lesssim \| \phi\|_{L^\infty}+|t|^{\frac14}\|\partial_x\phi\|_{L^2},
\end{aligned}
\end{equation}
These displacement bounds allow us to track the evolution of the squared amplitude and energy functionals. By \eqref{est:UtL2} and \eqref{est:UtLinfty}, we can bound the variation of the squares
\begin{equation}\label{est:UtsqL2}
\begin{aligned}
\big\| |U(t)\phi|^2-|\phi|^2\big\|_{L^2}&\lesssim \Big( \| \phi\|_{L^\infty}+|t|^{\frac14}\| \partial_x\phi\|_{L^2}\Big) |t|^{\frac12}\| \partial_x\phi\|_{L^2}, \\
\big\| |U(t)\phi|^2-\rho^2\big\|_{L^2}&\lesssim \big\| |\phi|^2-\rho^2\big\|_{L^2}+\Big( \| \phi\|_{L^\infty}+|t|^{\frac14}\| \partial_x\phi\|_{L^2}\Big) |t|^{\frac12}\| \partial_x\phi\|_{L^2}.
\end{aligned}
\end{equation}
Since $U(t)$ acts unitarily on $H^s(\R)$, the estimate \eqref{est:UtsqL2} guarantees that the linear evolution preserves the space $\mathcal{Z}^k_\rho$. Specifically, for any $\rho \geq 0$ and $\phi \in \mathcal{Z}^k_\rho$, we have $U(t)\phi \in C(\R;\mathcal{Z}^k_\rho)$ with the energy bounds
\begin{equation}\label{est:UtEk}
\begin{aligned}  
\| U(t)\phi\|_{X^k}{}\lesssim \| \phi\|_{X^k}+|t|^{\frac14}\| \partial_x\phi\|_{L^2}, \quad \text{and}\quad
\widetilde{\mathcal{E}}(U(t)\phi){}\lesssim \widetilde{\mathcal{E}}(\phi)+(1+|t|^{\frac34})\| \phi\|_{X^1}^2.
\end{aligned} 
\end{equation}
Finally, to establish continuity with respect to the initial data, we estimate the distance between two evolving states $\phi, \psi \in \mathcal{Z}^{1}$. Applying the algebraic identity \eqref{id:sqdiff} with the operator $A = U(t)$ and utilizing the bounds \eqref{est:UtL2} and \eqref{est:UtLinfty}, we obtain
\begin{equation}\label{est:d-Ut}
\begin{aligned}
\widetilde{d}(U(t)\phi, U(t)\psi)
&={}\big\| (|U(t)\phi|^2-\rho_\phi^2)-(|U(t)\psi|^2-\rho_\psi^2) \big\|_{L^2} \\
&\lesssim 
\big\| (|\phi|^2-\rho_\phi^2)-(|\psi|^2-\rho_\psi^2)\big\|_{L^2}  
+\big\| U(t)(\phi-\psi)-(\phi-\psi)\big\|_{L^2} \| U(t)\phi \|_{L^\infty} \\
&\quad +\| U(t)\psi -\psi\|_{L^2}\| U(t)(\phi-\psi)\|_{L^\infty} 
+\| \phi-\psi\|_{L^\infty}\| U(t)\phi-\phi\|_{L^2} \\
&\quad +\| \psi\|_{L^\infty}\big\| U(t)(\phi-\psi)-(\phi-\psi)\big\|_{L^2} \\
&\lesssim \widetilde{d}(\phi, \psi) +(1+|t|^{\frac34})\big( \| \phi\|_{X^1}+\| \psi\|_{X^1}\big)\|\phi -\psi\|_{X^1}.
\end{aligned}
\end{equation}
Furthermore, the base $X^k$ norm satisfies
\begin{equation}\label{est:UtXk-diff}
    \| U(t)\phi-U(t)\psi\|_{X^k}\lesssim \| \phi-\psi\|_{X^k}+|t|^{\frac14}\| \phi-\psi\|_{X^1}. 
\end{equation}
Together, \eqref{est:d-Ut} and \eqref{est:UtXk-diff} confirm that the mapping $\phi\mapsto U(t)\phi$ is continuous in the metric space $\mathcal{Z}^k$.


\subsection{Nonlinear estimates and an integration by parts lemma}

We record several crucial nonlinear bounds for the terms defined in \eqref{def:N0N1}.
\begin{lemma}\label{lem:N0N1}
Let $\delta>0$ and $\alpha, \beta \in \mathbb{R}$. The following estimates hold with implicit constants that are independent of $\delta$.
\begin{enumerate}
\item[$\mathrm{(i)}$] For any integer $k\geq 1$ and $f\in \mathcal{Z}^k$,
\begin{align*}
\| \mathcal{N}_\delta(f)\|_{L^2}&\lesssim \delta^{-1}\widetilde{\mathcal{E}}(f)\| f\|_{L^\infty}, \\
\| \partial_x^k\mathcal{N}_\delta(f)\|_{L^2}&\lesssim_k \delta^{-1}\Big( \big[ \widetilde{\mathcal{E}}(f)\big]^{\frac12}\| f\|_{X^1}+\| f\|_{X^1}^2\Big) \| \partial_xf\|_{H^{k-1}}.
\end{align*}
Consequently, 
\begin{align*}
\| \mathcal{N}_\delta(f)\|_{L^\infty}&\lesssim \delta^{-1}\Big( \big[ \widetilde{\mathcal{E}}(f)\big] ^{\frac34} \| f\|_{X^1}^{\frac32}+\big[ \widetilde{\mathcal{E}}(f)\big] ^{\frac12}\| f\|_{X^1}^2\Big), \\
\| \mathcal{N}_\delta(f)\|_{X^1}&\lesssim \delta^{-1}\Big( \big[ \widetilde{\mathcal{E}}(f)\big] ^{\frac34} \| f\|_{X^1}^{\frac32}+\| f\|_{X^1}^3\Big).
\end{align*}
\item[$\mathrm{(ii)}$] For any integer $k\geq 0$ and $f\in \mathcal{Z}^{k+1}$, 
\[ \| \partial_x^k\mathcal{N}_\infty(f)\|_{L^2}\lesssim_k \| f\|_{X^1}^2\| \partial_xf\|_{H^k}.\]
In particular, we have the full Sobolev bound:
\[ \| \mathcal{N}_\infty(f)\|_{H^k}\lesssim_k \| f\|_{X^1}^2\| f\|_{X^{k+1}}.\]
\item[$\mathrm{(iii)}$] For any $f_1,f_2\in \mathcal{Z}^1$, 
\begin{align*}
\| \mathcal{N}_\delta(f_1)-\mathcal{N}_\delta(f_2)\|_{L^2}&\lesssim \delta^{-1}\Big( \widetilde{\mathcal{E}}(f_1)\| f_1-f_2\|_{L^\infty} +\| f_2\|_{L^\infty} \widetilde{d}(f_1,f_2)\Big) , \\
\| \partial_x(\mathcal{N}_\delta(f_1)-\mathcal{N}_\delta(f_2))\|_{L^2}&\lesssim \delta^{-1}\Big\{ \Big( \big[ \widetilde{\mathcal{E}}(f_1)\big] ^{\frac12} \| f_1\|_{X^1}+\| f_1\|_{X^1}^2+\| f_2\|_{X^1}^2\Big)\| f_1-f_2\|_{X^1}\\
&\quad\qquad +\| f_2\|_{X^1}\big( \| f_1\|_{X^1}+\| f_2\|_{X^1}\big)^{\frac12}\| f_1-f_2\|_{X^1}^{\frac12}\big[ \widetilde{d}(f_1,f_2)\big] ^{\frac12} \Big\} ,\\
\| \mathcal{N}_\infty(f_1)-\mathcal{N}_\infty(f_2)\|_{L^2}&\lesssim
\big( \| f_1\|_{X^1}^2+\| f_2\|_{X^1}^2\big) \| f_1-f_2\|_{X^1}.
\end{align*}
\end{enumerate}
\end{lemma}
\begin{proof}
These estimates follow mechanically from applying H\"older's inequality, the Gagliardo-Nirenberg inequality, the Plancherel bound \eqref{est:Ldelta}, the Zhidkov energy identity \eqref{est:TLonZ}, and the uniform $L^2(\R)$ boundedness of the Hilbert transform $H$.

As a representative example, we demonstrate the second estimate in (iii). Applying the product rule and triangle inequality yields
\begin{align*}
\| \partial_x(\mathcal{N}_\delta(f_1)-\mathcal{N}_\delta(f_2))\|_{L^2} 
&\lesssim \| \partial_x(f_1-f_2)\|_{L^2}\| \mathcal{L}_\delta|f_1|^2\|_{L^\infty}+\| f_1-f_2\|_{L^\infty}\| \mathcal{L}_\delta\partial_x|f_1|^2\|_{L^2} \\
&\quad +\| \partial_xf_2\|_{L^2}\| \mathcal{L}_\delta(|f_1|^2-|f_2|^2)\|_{L^\infty} +\| f_2\|_{L^\infty}\| \mathcal{L}_\delta\partial_x(|f_1|^2-|f_2|^2)\|_{L^2} .
\end{align*}
To bound the $\mathcal{L}_\delta$ terms, we exploit the identity $\mathcal{L}_{\delta}(|f_{1}|^{2}-|f_{2}|^{2}) = (\Td -H)\partial_{x}\big\{ (|f_{1}|^{2}-\rho_{f_1}^{2})-(|f_{2}|^{2}-\rho_{f_2}^{2}) \big\}$. This allows us to apply \eqref{est:Ldelta} and Gagliardo-Nirenberg,
\begin{align*}
&\| \partial_x(\mathcal{N}_\delta(f_1)-\mathcal{N}_\delta(f_2))\|_{L^2} \\
&\lesssim \| f_1-f_2\|_{X^1}\delta^{-1}\Big( \big\| |f_1|^2-\rho_{f_1}^2\big\|_{L^2}^{\frac12}\| \partial_x|f_1|^2\|_{L^2}^{\frac12}+\| \partial_x|f_1|^2\|_{L^2}\Big) \\
&\quad +\| f_2\|_{X^1}\delta^{-1}\Big( \big\| (|f_1|^2-\rho_{f_1}^2)-(|f_2|^2-\rho_{f_2}^2)\big\|_{L^2}^{\frac12}\| \partial_x(|f_1|^2-|f_2|^2)\|_{L^2}^{\frac12}+\| \partial_x(|f_1|^2-|f_2|^2)\|_{L^2}\Big) \\
&\lesssim \delta^{-1}\Big( \big[ \widetilde{\mathcal{E}}(f_1)\big]^{\frac12}\|f \|_{X^1} +\| f_1 \|_{X^1}^2 \Big) \| f_1-f_2\|_{X^1} \\
&\quad + \delta^{-1} \| f_2\|_{X^1}
 \Big( \big[ \widetilde{d}(f_1,f_2)\big]^{\frac12} \big\{ ( \| f_1\|_{X^1} +  \| f_2\|_{X^1}) \| f_1-f_2 \|_{X^1}\big\}^{\frac12} + (\| f_1\|_{X^1}+\| f_2\|_{X^1})\| f_1-f_2\|_{X^1}\Big) .
\end{align*}

For the higher-order bounds ($k\geq 2$) in the second estimate of (i) and the first estimate of (ii), we use standard interpolation inequalities to control intermediate derivatives $\| \partial_x^jf\|_{L^2}$ and $\| \partial_x^jf\|_{L^\infty}$ (where $1\leq j\leq k-1$) using $\| \partial_xf\|_{L^2}$ and $\| \partial_x^kf\|_{L^2}$. The remaining estimates are proved analogously.
\end{proof}

We conclude this section with an algebraic lemma governing integration by parts, which is pivotal for the energy method.
\begin{lemma}\label{lem:ibp}
Let $f\in W^{1,\infty}(\R)$ and $g\in H^1(\R)$. Then:
\begin{align*}
&\big| \Re \big( f(\overline{f}\partial_xg),g \big)_{L^2} \big| + \big| \Re \big( f(f\partial_x\overline{g}),g\big)_{L^2} \big| + \big| \Re \big( f(iH)(\overline{f}\partial_xg),g\big)_{L^2} \big| \lesssim \| f\|_{L^\infty}\| \partial_xf\|_{L^\infty}\| g\|_{L^2}^2.
\end{align*}
\end{lemma}
\begin{remark}
Crucially, the right-hand side depends only on the $L^2$-norm of $g$, absorbing the derivative $\partial_x g$. However, the standard energy method fails to provide a comparable absorption for the term $\Re \big( f(iH)(f\partial_x\overline{g}),g\big)_{L^2}$. To control interactions of this specific form, we must resort to a modified energy method later in the paper.
\end{remark}

\begin{proof}[Proof of Lemma~\ref{lem:ibp}]
The first two terms are easily treated by integration by parts (one can estimate the second term without taking the real part).
For the term involving $H$, we rewrite it as
\[ \Re \big( f(iH)(\overline{f}\partial_xg),g\big)_{L^2} = \Re \big( iH\partial_x(\overline{f}g),\overline{f}g\big)_{L^2} - \Re \big( f(iH)(g\partial_x\overline{f}),g\big)_{L^2}.\]
Observe that the first term on the right-hand side vanishes by the positivity of $H\partial_x$. The remaining term is easily bounded by applying the Cauchy-Schwarz inequality along with the uniform boundedness of $H$ on $L^2(\R)$.
\end{proof}


\section{Local well-posedness}
\label{sec3}

In this section, we prove our main local well-posedness result, Theorem~\ref{thm:LWP}.  We focus on the case $\delta \geq 1$ and establish local well-posedness uniformly in $\delta$ for initial data $\phi\in \mathcal{Z}^2$ satisfying $\mathcal{E}^2_\delta(\phi) \leq M_0^2$. The argument easily adapts to general $\delta > 0$ if we permit various constants to depend on $\delta$.

At the end of the section, we will also provide the proof of Theorem~\ref{thm:conv}. Throughout, we fix $\alpha, \beta \in \R$ and suppress the dependence of constants on these parameters.

We proceed via a standard compactness argument. For each frequency truncation parameter $\ell \in \Z_{\geq 0}$, we consider the mollified equation
\begin{equation}\label{gINLSmoll}
i\partial_t u_\ell - \partial_x^2 u_\ell = P_{\leq \ell}\mathcal{N}(P_{\leq \ell}u_\ell),
\end{equation}
which corresponds to the integral equation
\begin{equation}\label{gINLSmoll-int}
u_\ell(t) = U(t)u_\ell(0) - i\int_0^t U(t-s)P_{\leq \ell}\mathcal{N}(P_{\leq \ell}u_\ell(s))\,ds,
\end{equation}
where we recall $U(t):=e^{-it\partial_x^2}\,$ and  $u_\ell(0) = P_{\leq \ell}\phi$. 

The proof proceeds in four main steps. 
First, in Proposition~\ref{prop:LWPmoll}, we establish the existence of a unique local solution  $u_\ell \in C([0,T_\ell); \mathcal{Z}^\infty)$ to the mollified problem with initial data $u_\ell(0) = P_{\leq \ell}\phi$. Second, in Proposition~\ref{prop:apriori-Zk}, we derive a uniform-in-$\ell$ time of existence $T_0 > 0$ and an a priori $\mathcal{Z}^2$ bound for the sequence $\{u_\ell\}$. Third, in Proposition~\ref{prop:apriori-d1}, we prove that $\{u_\ell\}$ is a Cauchy sequence in $C([0,T_*];\mathcal{Z}^1)$ for a possibly shorter time $T_* \leq T_0$, allowing us to construct a solution to \eqref{gINLS} by taking the limit $\ell \to \infty$. Finally, we verify the uniqueness and continuity of the flow map in $\mathcal{Z}^2$, completing the proof of Theorem~\ref{thm:LWP}.

\begin{proposition}\label{prop:LWPmoll}
Let {$\delta>0$ and} $\phi\in \mathcal{Z}^1$. 
 For each $\ell \in \Z_{\geq 0}$, the Cauchy problem \eqref{gINLSmoll} with initial data $u_\ell(0) = P_{\leq \ell}\phi$ admits a unique maximal solution $ u_\ell \in C([0,T_\ell); \mathcal{Z}^\infty). $ Furthermore, if the maximal time of existence $T_\ell$ is finite, the solution satisfies the blow-up criterion
 $$ \lim_{t\uparrow T_\ell} \mathcal{E}^1(u_\ell(t)) = \infty. $$
\end{proposition}

\begin{proof}
Because the initial data satisfies $\mathrm{supp}\,(\widehat{P_{\leq \ell}\phi}) \subset \{ |\xi| \leq 2^{\ell+1}\}$, any local solution constructed via standard iteration will preserve this frequency localization for all time, ensuring spatial smoothness, that is $u_\ell \in \mathcal{Z}^\infty$. It therefore suffices to prove local existence in $\mathcal{Z}^1$ for data satisfying $\mathcal{E}^1(\phi) \leq M_0^2$. We seek a fixed point for the Duhamel map
$$
    \Gamma[u](t) := U(t)P_{\leq \ell}\phi - i\int_0^t U(t-s)P_{\leq \ell}\mathcal{N}(P_{\leq \ell}u(s))\,ds. 
$$
For $0 < T \leq 1$, we define the complete metric space
$$ 
    X_T := \big\{ u \in C([0,T]; \mathcal{Z}^1) : \| \mathcal{E}^1(u)\|_{L^\infty_T} \leq 2c M_0^2 \big\}, \qquad d_T(u,v) := \| d^1(u,v)\|_{L^\infty_T},
$$
where $c$ is the constant from the linear evolution bound \eqref{est:UtEk}. Using the Littlewood-Paley bounds \eqref{est:EkPl} and the nonlinear estimates from Lemma~\ref{lem:N0N1}, the integral term is easily controlled in $H^1$. For any $u, v \in X_T$, we obtain the following bounds
\begin{align}
\Big\| \int_0^t U(t-s)P_{\leq \ell}\mathcal{N}(P_{\leq \ell}u)\,ds \Big\|_{L^\infty_T H^1} &\lesssim T(\delta^{-1} + 2^\ell)M_0^3, \label{est:LWPmoll1} \\
\Big\| \int_0^t U(t-s)P_{\leq \ell}\big[ \mathcal{N}(P_{\leq \ell}u) - \mathcal{N}(P_{\leq \ell}v) \big]\,ds \Big\|_{L^\infty_T H^1} &\lesssim T(\delta^{-1} + 2^\ell)(1 + M_0^3)d_T(u,v). \label{est:LWPmoll2}\end{align}
Applying the metric properties \eqref{est:Ekf+g} and \eqref{est:d1f+g}, these estimates imply that $\Gamma$ maps $X_T$ into itself and satisfies the contraction property
$$ 
    d_T(\Gamma[u], \Gamma[v]) \leq \frac{1}{2} d_T(u,v), 
$$
provided we choose $T > 0$ sufficiently small such that $T(\delta^{-1} + 2^\ell)(1 + M_0^4) \ll 1$. By the Banach fixed-point theorem, $\Gamma$ admits a unique local solution $u_\ell \in X_T$. 

Extending this local solution to the maximal interval $[0, T_\ell)$ and establishing the standard blow-up alternative follows from classical continuation arguments, which completes the proof.
\end{proof}


\begin{proposition}[Uniform $\mathcal{Z}^2$ estimate]\label{prop:apriori-Zk}

Let $M_0>0$ and $\delta\geq 1$. Suppose $\phi \in \mathcal{Z}^2$ is initial data satisfying the energy bound $\mathcal{E}^2_\delta(\phi)\leq M_0^2$. For each frequency $\ell \geq 0$, let $u_\ell \in C([0,T_\ell); \mathcal{Z}^\infty)$ be the unique maximal solution to the mollified problem \eqref{gINLSmoll} with initial data $u_\ell(0)=P_{\leq \ell}\phi$. Then, there exists a uniform time of existence $T_0 \in (0, 1]$, satisfying 
$$
    (1+M_0^4)^{-1} \lesssim T_0 \leq 1
$$
such that $T_\ell > T_0$ for all $\ell \geq 0$. Furthermore, on the interval $[0, T_0]$, the sequence $\{u_\ell\}$ satisfies the uniform $\mathcal{Z}^2$ energy bound
\begin{equation}\label{est:aprioriZ2}
\| \mathcal{E}^2_\delta(u_\ell) \|_{L^\infty_{T_0}} \lesssim \mathcal{E}^2_\delta(\phi).
\end{equation}
Additionally, we have the following lower-order control
\begin{align}\label{est:aprioriZ2'}
\| u_\ell \|_{L^\infty_{T_0}X^1}^2 \lesssim \| \phi \|_{X^1}^2, 
\quad \text{and}\quad
\| \widetilde{\mathcal{E}}(u_\ell) \|_{L^\infty_{T_0}} \lesssim \widetilde{\mathcal{E}}(\phi) + \| \phi \|_{X^1}^2. 
\end{align}
Moreover, for any integer $k \geq 3$, the higher-order regularities satisfy
\begin{equation}\label{est:aprioriZk}
\| \partial_x^k u_\ell \|_{L^\infty_{T_0}L^2}^2 \lesssim_k \| \partial_x^k P_{\leq \ell}\phi \|_{L^2}^2 + \| \phi \|_{X^1}^2.
\end{equation}
\end{proposition}

\begin{proof}
To establish the uniform bound \eqref{est:aprioriZ2}, the principal difficulty arises from a derivative loss in the highest-order energy estimate.
To overcome this, we employ a \emph{modified energy method}: we design a frequency-localized correction term that neutralizes the problematic nonlinear interactions in the nonlinearity of \eqref{gINLSmoll}.

Let $\ell\geq 0$, and take arbitrary $T\in (0,1]$ with $T<T_\ell$.

We begin by controlling $\|u_\ell\|_{X^1}$ and $\widetilde{\mathcal{E}}(u_\ell)$. Using the integral equation, the linear bound \eqref{est:UtEk}, the metric inequality \eqref{est:Ekf+g} with $k=1$, the Littlewood-Paley property $\| P_{\leq \ell}u_\ell \| \leq \| u_\ell \|$, and the nonlinear bounds from Lemma~\ref{lem:N0N1}, we can bound the Duhamel integral directly in terms of $u_\ell$
\begin{equation}\label{est:aprioriZ2-A}
\begin{aligned}
\| u_\ell \|_{L^\infty_TX^1}^2 
&\lesssim \| U(t)P_{\leq \ell}\phi \|_{L^\infty_TX^1}^2+ \Big\| \int_0^t U(t-s)P_{\leq \ell}\mathcal{N}(P_{\leq \ell}u_\ell(s))\,ds \Big\|_{L^\infty_TX^1}^2 \\
&\lesssim \| \phi \|_{X^1}^2+T^2\Big( \delta^{-1}\| u_\ell \|_{L^\infty_TX^1}^{\frac32}\| \widetilde{\mathcal{E}}(u_\ell) \|_{L^\infty_T}^{\frac34}+\| u_\ell \|_{L^\infty_TX^1}^2\| u_\ell \|_{L^\infty_TX^2}\Big)^2 \\
&\lesssim \| \phi \|_{X^1}^2 +T^2\| \mathcal{E}^2_\delta(u_\ell) \|_{L^\infty_T}^2\| u_\ell \|_{L^\infty_TX^1}^2 \,.
\end{aligned}
\end{equation}
Similarly, for the energy functional we have
\begin{equation}\label{est:aprioriZ2-A'}
\begin{aligned}
\| \widetilde{\mathcal{E}}(u_\ell) \|_{L^\infty_T}
&\lesssim \| \widetilde{\mathcal{E}}(U(t)P_{\leq \ell}\phi) \|_{L^\infty_T} +\Big( \| U(t)P_{\leq \ell}\phi \|_{L^\infty_TL^\infty} + \Big\| \int_0^t U(t-s)P_{\leq \ell}\mathcal{N}(P_{\leq \ell}u_\ell)\,ds \Big\|_{L^\infty_TL^\infty}\Big) \\
&\hspace{4cm} \times \Big\| \int_0^t U(t-s)P_{\leq \ell}\mathcal{N}(P_{\leq \ell}u_\ell)\,ds \Big\|_{L^\infty_TL^2} \\
&\lesssim \widetilde{\mathcal{E}}(\phi)+\| \phi \|_{X^1}^2+\Big\{ \| \phi \|_{X^1}+T\Big( \delta^{-1}\| u_\ell \|_{L^\infty_TX^1}^{\frac32}\| \widetilde{\mathcal{E}}(u_\ell) \|_{L^\infty_T}^{\frac34}+\| u_\ell \|_{L^\infty_TX^1}^2\| u_\ell \|_{L^\infty_TX^2}\Big) \Big\} \\
&\hspace{4cm} \times T\Big( \delta^{-1}\| u_\ell \|_{L^\infty_TX^1}\| \widetilde{\mathcal{E}}(u_\ell)\|_{L^\infty_T}+\| u_\ell \|_{L^\infty_TX^1}^2\| u_\ell \|_{L^\infty_TX^2}\Big) \\
&\lesssim \widetilde{\mathcal{E}}(\phi)+\| \phi \|_{X^1}^2+ T^2\delta^{-\frac23}\| \mathcal{E}^2_\delta(u_\ell) \|_{L^\infty_T}^2\| \widetilde{\mathcal{E}}(u_\ell) \|_{L^\infty_T} +T^2\| \mathcal{E}^2_\delta(u_\ell) \|_{L^\infty_T}^2\| u_\ell \|_{L^\infty_TX^1}^2.
\end{aligned}
\end{equation}

We now estimate the highest-order component of $\mathcal{E}^2(u_\ell)$. Since $u_\ell$ is a smooth solution, one computes using equation \eqref{gINLSmoll},
$$\begin{aligned}
\frac{d}{dt}\| \partial_x^2u_\ell(t) \|_{L^2}^2 &= 2\Re \big( \partial_x^2\partial_t u_\ell(t),\partial_x^2u_\ell(t)\big)_{L^2} = 2\Re \big( {-}i \partial_x^2\mathcal{N}(P_{\leq \ell}u_\ell(t)),\partial_x^2P_{\leq \ell}u_\ell(t)\big)_{L^2}  \\
&= 2\Re \big( {-}i\partial_x^2\mathcal{N}_\delta(P_{\leq \ell}u_\ell(t)) ,\partial_x^2P_{\leq \ell}u_\ell(t)\big)_{L^2} \\
&\quad + 2\Re \big( \partial_x^2\big( P_{\leq \ell}u_\ell(t)(\alpha +i\beta H)\partial_x |P_{\leq \ell}u_\ell(t)|^2\big) ,\partial_x^2P_{\leq \ell}u_\ell(t)\big)_{L^2}.
\end{aligned} $$
The first term is harmless and bounded directly via Lemma~\ref{lem:N0N1}
$$\Big| \big( {-}i\partial_x^2\mathcal{N}_\delta(P_{\leq \ell}u_\ell(t)) ,\partial_x^2P_{\leq \ell}u_\ell(t)\big)_{L^2} \Big| \lesssim \delta^{-1} \Big( \big[ \widetilde{\mathcal{E}}(u_\ell(t))\big] ^{\frac12}\|u_\ell(t)\|_{X^1}+\| u_\ell(t) \|_{X^1}^2\Big) \| \partial_xu_\ell(t) \|_{H^1}^2. $$
For the second term, distributing the differentials allows us to safely estimate interactions where no function carries three derivatives, that is,
$$\begin{aligned}
&\Big| 2\Re \big( \partial_x^2\big( P_{\leq \ell}u_\ell(t)(\alpha +i\beta H)\partial_x |P_{\leq \ell}u_\ell(t)|^2\big) ,\partial_x^2P_{\leq \ell}u_\ell(t)\big)_{L^2} \\
&\quad -2\Re \big( P_{\leq \ell}u_\ell(t)(\alpha +i\beta H)\big(\partial_x^3P_{\leq \ell}u_\ell(t)\,\overline{P_{\leq \ell}u_\ell(t)}+P_{\leq \ell}u_\ell(t)\partial_x^3\overline{P_{\leq \ell}u_\ell(t)}\big) ,\partial_x^2P_{\leq \ell}u_\ell(t)\big)_{L^2} \Big| \\
&\lesssim \| u_\ell(t) \|_{X^2}^2\| \partial_xu_\ell(t) \|_{H^1}^2.
\end{aligned} $$Even the terms with three derivatives can be controlled via integration by parts and Lemma~\ref{lem:ibp}, except for one specific interaction
$$\begin{aligned}
&\Big| 2\Re \big( P_{\leq \ell}u_\ell(t)(\alpha +i\beta H)\big(\partial_x^3P_{\leq \ell}u_\ell(t)\,\overline{P_{\leq \ell}u_\ell(t)}+P_{\leq \ell}u_\ell(t)\partial_x^3\overline{P_{\leq \ell}u_\ell(t)}\big) ,\partial_x^2P_{\leq \ell}u_\ell(t)\big)_{L^2} \\
&\quad -2\beta\Re \big( P_{\leq \ell}u_\ell(t)(iH)\big( P_{\leq \ell}u_\ell(t)\partial_x^3\overline{P_{\leq \ell}u_\ell(t)}\big) ,\partial_x^2P_{\leq \ell}u_\ell(t)\big)_{L^2} \Big| \\
&\lesssim \| u_\ell(t) \|_{X^2}^2\| \partial_xu_\ell(t) \|_{H^1}^2.
\end{aligned} $$
Summarizing these bounds, we arrive at the differential inequality
\begin{equation}\label{est:aprioriZ2-1}
\begin{aligned}
\frac{d}{dt}\| \partial_x^2u_\ell(t) \|_{L^2}^2 &\leq 2\beta\Re \big( P_{\leq \ell}u_\ell(t)(iH)\big( P_{\leq \ell}u_\ell(t)\partial_x^3\overline{P_{\leq \ell}u_\ell(t)}\big) ,\partial_x^2P_{\leq \ell}u_\ell(t)\big)_{L^2} \\
&\quad +C\Big( \delta^{-1}\big[ \widetilde{\mathcal{E}}(u_\ell(t))\big] ^{\frac12}\| u_\ell(t) \|_{X^2}+\| u_\ell(t) \|_{X^2}^2\Big) \| \partial_xu_\ell(t) \|_{H^1}^2.\end{aligned}\end{equation}

\medskip
To handle the derivative loss arising from the first term on the right-hand side of \eqref{est:aprioriZ2-1}, we introduce a time-dependent corrector $\mathcal{I}_2(t)$ designed precisely to cancel this interaction. Letting $n_0\in \Z_{\geq 0}$ be a frequency threshold to be determined later, we define:
\begin{equation}\label{def:I2}
\mathcal{I}_2(t):=\beta\Re \big( P_{\leq \ell}u_\ell(t)H\partial_x\big(P_{\leq \ell}u_\ell(t)\partial_x\overline{P_{\geq n_0}P_{\leq \ell}u_\ell(t)}\big) ,\partial_xP_{\geq n_0}P_{\leq \ell}u_\ell(t)\big)_{L^2} .
\end{equation}
By Bernstein's inequality, this corrector satisfies the bound:
\begin{equation}\label{bd:I2}
\begin{aligned}
|\mathcal{I}_2(t)|
&\lesssim \| P_{\leq \ell}u_\ell(t) \|_{L^\infty}\| \partial_xP_{\geq n_0}P_{\leq \ell}u_\ell(t) \|_{L^2} \| P_{\leq \ell}u_\ell(t)\partial_xP_{\geq n_0}P_{\leq \ell}u_\ell(t) \|_{H^1} \\
&\lesssim 2^{-n_0}\| u_\ell(t) \|_{X^2}^2\| \partial_x^2u_\ell(t) \|_{L^2}^2.
\end{aligned}
\end{equation}
Differentiating $\mathcal{I}_2(t)$ with respect to time, applying equation \eqref{gINLSmoll}, and utilizing the symmetry property \eqref{id:Tdeltasymm} for $H\partial_x$, we obtain
$$
    \begin{aligned}
    \frac{d}{dt}\mathcal{I}_2(t)&=2\beta \Re \big( H\partial_x\big(\partial_tP_{\leq \ell}u_\ell(t)\cdot \partial_x\overline{P_{\geq n_0}P_{\leq \ell}u_\ell(t)}\big) ,\overline{P_{\leq \ell}u_\ell(t)}\partial_xP_{\geq n_0}P_{\leq \ell}u_\ell(t)\big)_{L^2} \\
    &\quad +2\beta \Re \big( H\partial_x\big(P_{\leq \ell}u_\ell(t)\cdot \partial_x\overline{P_{\geq n_0}\partial_tP_{\leq \ell}u_\ell(t)}\big) ,\overline{P_{\leq \ell}u_\ell(t)}\partial_xP_{\geq n_0}P_{\leq \ell}u_\ell(t)\big)_{L^2} \\
    &=2\beta \Re \big( iH\partial_x\big(P_{\leq \ell}u_\ell(t) \partial_x^3\overline{P_{\geq n_0}P_{\leq \ell}u_\ell(t)}\big) ,\overline{P_{\leq \ell}u_\ell(t)}\partial_xP_{\geq n_0}P_{\leq \ell}u_\ell(t)\big)_{L^2} +\mathcal{R}(t).
\end{aligned}
$$
Here, the remainder $\mathcal{R}(t)$ collects all terms with at most two derivatives on each function$$\begin{aligned}
    \mathcal{R}(t)&=-2\beta \Re\big(  {-}iH\big(\partial_x^2P_{\leq \ell}u_\ell(t)\cdot \partial_x\overline{P_{\geq n_0}P_{\leq \ell}u_\ell(t)}\big) ,\partial_x\big( \overline{P_{\leq \ell}u_\ell(t)}\partial_xP_{\geq n_0}P_{\leq \ell}u_\ell(t)\big) \big)_{L^2} \\
    &\quad +2\beta \Re \big( {-}iH\partial_x\big(P_{\leq \ell}\mathcal{N}(P_{\leq \ell}u_\ell(t))\cdot \partial_x\overline{P_{\geq n_0}P_{\leq \ell}u_\ell(t)}\big) ,\overline{P_{\leq \ell}u_\ell(t)}\partial_xP_{\geq n_0}P_{\leq \ell}u_\ell(t)\big)_{L^2} \\
    &\quad -2\beta \Re \big( iH\big(P_{\leq \ell}u_\ell(t)\cdot \partial_x\overline{P_{\geq n_0}P_{\leq \ell}\mathcal{N}(P_{\leq \ell}u_\ell(t))}\big) ,\partial_x\big( \overline{P_{\leq \ell}u_\ell(t)}\partial_xP_{\geq n_0}P_{\leq \ell}u_\ell(t)\big) \big)_{L^2} .
\end{aligned}$$
Using Lemma~\ref{lem:N0N1}, this remainder is controlled by
$$
    |\mathcal{R}(t)|\lesssim \Big( \| u_\ell(t) \|_{X^2}^2+\delta^{-1}\big[ \widetilde{\mathcal{E}}(u_\ell(t))\big] ^{\frac34}\| u_\ell(t) \|_{X^2}^{\frac52}+\| u_\ell(t) \|_{X^2}^4\Big) \| \partial_xu_\ell(t) \|_{H^1}^2. 
$$
Crucially, the main term in $\frac{d}{dt}\mathcal{I}_2(t)$ compensate the first term in \eqref{est:aprioriZ2-1}, thus by integrating by parts and applying Bernstein's inequality $\| \partial_x^3P_{\leq n_0-1}P_{\leq \ell}u_\ell \|_{L^2}\lesssim 2^{n_0}\| \partial_x^2u_\ell \|_{L^2}$, the interaction between the main terms satisfies
$$
\begin{aligned}
&\Big| 2\beta \Re \big( iH\partial_x\big(P_{\leq \ell}u_\ell(t) \partial_x^3\overline{P_{\geq n_0}P_{\leq \ell}u_\ell(t)}\big) ,\overline{P_{\leq \ell}u_\ell(t)}\partial_xP_{\geq n_0}P_{\leq \ell}u_\ell(t)\big)_{L^2} \\
&\quad + 2\beta \Re \big( iH\big(P_{\leq \ell}u_\ell(t)\partial_x^3\overline{P_{\leq \ell}u_\ell(t)}\big) ,\overline{P_{\leq \ell}u_\ell(t)}\partial_x^2P_{\leq \ell}u_\ell(t)\big)_{L^2} \Big| \\
&\leq \Big| 2\beta \Re \big( iH\big(P_{\leq \ell}u_\ell(t)\partial_x^3\overline{P_{\geq n_0}P_{\leq \ell}u_\ell(t)}\big) ,\partial_x\overline{P_{\leq \ell}u_\ell(t)}\cdot \partial_xP_{\geq n_0}P_{\leq \ell}u_\ell(t)\big)_{L^2} \Big| \\
&\quad +\Big| 2\beta \Re \big( iH\big(P_{\leq \ell}u_\ell(t)\partial_x^3\overline{P_{\geq n_0}P_{\leq \ell}u_\ell(t)}\big) ,\overline{P_{\leq \ell}u_\ell(t)}\partial_x^2P_{\geq n_0}P_{\leq \ell}u_\ell(t)\big)_{L^2} \\
&\quad\quad - 2\beta \Re \big( iH\big(P_{\leq \ell}u_\ell(t)\partial_x^3\overline{P_{\leq \ell}u_\ell(t)}\big) ,\overline{P_{\leq \ell}u_\ell(t)}\partial_x^2P_{\leq \ell}u_\ell(t)\big)_{L^2} \Big| \\
&\lesssim (1+2^{n_0})\| u_\ell(t) \|_{X^2}^2\| \partial_xu_\ell(t) \|_{H^1}^2.
\end{aligned}$$
This yields the final bound for the corrector derivative
\begin{equation}\label{est:aprioriZ2-2}
\begin{aligned}
\frac{d}{dt}\mathcal{I}_2(t)&\leq - 2\beta \Re \big( iH\big(P_{\leq \ell}u_\ell(t)\partial_x^3\overline{P_{\leq \ell}u_\ell(t)}\big) ,\overline{P_{\leq \ell}u_\ell(t)}\partial_x^2P_{\leq \ell}u_\ell(t)\big)_{L^2} \\
&\quad +C\Big( 2^{n_0}\| u_\ell(t) \|_{X^2}^2+ \delta^{-1}\big[ \widetilde{\mathcal{E}}(u_\ell(t))\big] ^{\frac34}\| u_\ell(t) \|_{X^2}^{\frac52}+\| u_\ell(t) \|_{X^2}^4\Big) \| \partial_xu_\ell(t) \|_{H^1}^2.
\end{aligned}
\end{equation}

Combining \eqref{est:aprioriZ2-1} and \eqref{est:aprioriZ2-2}, we obtain the modified energy inequality,
$$\begin{aligned}
\frac{d}{dt}\big( \| \partial_x^2u_\ell(t) \|_{L^2}^2+\mathcal{I}_2(t)\big) &\lesssim 2^{n_0}\| u_\ell(t) \|_{X^2}^4+\| u_\ell(t) \|_{X^2}^6 \\
&\quad +\delta^{-1}\big[ \widetilde{\mathcal{E}}(u_\ell(t))\big] ^{\frac12}\| u_\ell(t) \|_{X^2}^3+\delta^{-1}\big[ \widetilde{\mathcal{E}}(u_\ell(t))\big] ^{\frac34}\| u_\ell(t) \|_{X^2}^{\frac92}.
\end{aligned}$$
Integrating this over $[0, t]$ and applying the corrector bound \eqref{bd:I2}, we obtain
\begin{equation}\label{est:aprioriZ2-C}
\begin{aligned}
\| \partial_x^2u_\ell \|_{L^\infty_TL^2}^2
&\leq \| \partial_x^2\phi \|_{L^2}^2+C
\Big( \begin{aligned}[t]
&(2^{-n_0}+T2^{n_0})\| u_\ell \|_{L^\infty_TX^2}^4+T\| u_\ell \|_{L^\infty_TX^2}^6 \\
&+T\delta^{-1}\| \widetilde{\mathcal{E}}(u_\ell) \|_{L^\infty_T}^{\frac12}\| u_\ell \|_{L^\infty_TX^2}^3+T\delta^{-1}\| \widetilde{\mathcal{E}}(u_\ell) \|_{L^\infty_T}^{\frac34}\| u_\ell \|_{L^\infty_TX^2}^{\frac92}\Big)
\end{aligned} \\
&\lesssim \| \partial_x^2\phi \|_{L^2}^2+\Big( (2^{-n_0}+T2^{n_0})\| \mathcal{E}^2_\delta(u_\ell) \|_{L^\infty_T}+T\| \mathcal{E}^2_\delta(u_\ell) \|_{L^\infty_T}^2\Big) \| u_\ell \|_{L^\infty_TX^2}^2.
\end{aligned}
\end{equation}
Putting \eqref{est:aprioriZ2-A}, \eqref{est:aprioriZ2-A'}, and \eqref{est:aprioriZ2-C} together,  yields the full energy  bound
\begin{equation}\label{est:aprioriZ2-D}
\| \mathcal{E}^2_\delta(u_\ell) \|_{L^\infty_T}\lesssim \mathcal{E}^2_\delta(\phi) +\Big( (2^{-n_0}+T2^{n_0}) \| \mathcal{E}^2_\delta(u_\ell) \|_{L^\infty_T} +T\| \mathcal{E}^2_\delta(u_\ell) \|_{L^\infty_T}^2\Big) \|\mathcal{E}^2_\delta(u_\ell) \|_{L^\infty_T} .
\end{equation}
To prove the main bounds, we select the frequency threshold $2^{n_0}=C(1+M_0^2)$ for a large constant $C>0$ so that $2^{-n_0}\mathcal{E}^2_\delta(\phi)\leq 2^{-n_0}M_0^2\ll 1$. Then, we choose $T_0$ sufficiently small so that $T_0(1+M_0^4)\ll 1$. A standard continuity argument applied to \eqref{est:aprioriZ2-D} then forces
$$\| \mathcal{E}^2_\delta(u_\ell) \|_{L^\infty_T}\lesssim \mathcal{E}^2_\delta(\phi),\qquad 0<T<\min \{ T_0,T_\ell\}. $$
By the blow-up criterion in Proposition~\ref{prop:LWPmoll}, this guarantees $T_\ell>T_0$ and establishes \eqref{est:aprioriZ2}. Plugging this back into \eqref{est:aprioriZ2-A} and \eqref{est:aprioriZ2-A'} (modifying $T_0$ if necessary) immediately yields both of the lower-order estimates listed in \eqref{est:aprioriZ2'}.

Finally, the bounds for $k \geq 3$ in \eqref{est:aprioriZk} follow from an identical modified energy approach. For the highest component $\| \partial_x^ku_\ell(t) \|_{L^2}$, one uses Lemmas~\ref{lem:N0N1}, \ref{lem:ibp}, and interpolation to derive
\begin{equation}\label{est:aprioriZk-1}
\begin{aligned}
\frac{d}{dt}\| \partial_x^ku_\ell(t) \|_{L^2}^2 &\leq 2\beta\Re \big( P_{\leq \ell}u_\ell(t)(iH)\big( P_{\leq \ell}u_\ell(t)\partial_x^{k+1}\overline{P_{\leq \ell}u_\ell(t)}\big) ,\partial_x^kP_{\leq \ell}u_\ell(t)\big)_{L^2} \\
&\quad +C_k\Big( \delta^{-1}\big[ \widetilde{\mathcal{E}}(u_\ell(t))\big] ^{\frac12}\| u_\ell(t) \|_{X^2}+\| u_\ell(t) \|_{X^2}^2\Big)\| \partial_xu_\ell(t) \|_{H^{k-1}}^2.
\end{aligned}
\end{equation}

Constructing the $k$-th order corrector 
$$
    \mathcal{I}_k(t) := \beta\Re \big( P_{\leq \ell}u_\ell(t)H\partial_x\big(P_{\leq \ell}u_\ell(t)\partial_x^{k-1}\overline{P_{\geq n_1}P_{\leq \ell}u_\ell(t)}\big) ,\partial_x^{k-1}P_{\geq n_1}P_{\leq \ell}u_\ell(t)\big)_{L^2}
$$
for $n_1 \gg 1$ yields corresponding derivative cancellations and bounds analogous to $\mathcal{I}_2(t)$:
\begin{align*}
|\mathcal{I}_k(t)|
&\lesssim_k 2^{-n_1}\| u_\ell(t) \|_{X^2}^2\| \partial_x^k u_\ell(t) \|_{L^2}^2, \\
\frac{d}{dt}\mathcal{I}_k(t)&=- 2\beta \Re \big( iH\big(P_{\leq \ell}u_\ell(t)\partial_x^{k+1}\overline{P_{\leq \ell}u_\ell(t)}\big) ,\overline{P_{\leq \ell}u_\ell(t)}\partial_x^kP_{\leq \ell}u_\ell(t)\big)_{L^2} \\
&\quad +C_k\Big\{ 2^{n_1}\| u_\ell(t)\|_{X^2}^2+ \delta^{-1}\big[ \widetilde{\mathcal{E}}(u_\ell(t))\big] ^{\frac34}\| u_\ell(t)\|_{X^2}^{\frac52}+\| u_\ell(t)\|_{X^2}^4\Big\} \| \partial_xu_\ell(t)\|_{H^{k-1}}^2
\end{align*}
in analogy with \eqref{bd:I2} and \eqref{est:aprioriZ2-2}.
Setting $2^{n_1} = C(1+M_0^2)$ with $C\gg _k 1$ allows us to deduce $|\mathcal{I}_k(t)| \leq \frac{1}{2} \| \partial_x^k u_\ell(t) \|_{L^2}^2$ for $t\in[0, T_0]$ and the modified energy inequality
$$\frac{d}{dt}\big( \| \partial_x^ku_\ell(t) \|_{L^2}^2+\mathcal{I}_k(t)\big) \lesssim_k (M_0^2+M_0^4)\| \partial_xu_\ell(t) \|_{H^{k-1}}^2,\qquad t\in [0,T_0]. $$Since \eqref{est:aprioriZ2'} ensures $\| \partial_xu_\ell(t) \|_{H^{k-1}}^2 \lesssim \| \phi \|_{X^1}^2 + \| \partial_x^ku_\ell(t) \|_{L^2}^2$, the functional $X_k(t) := \| \partial_x^ku_\ell(t) \|_{L^2}^2+\mathcal{I}_k(t)+\| \phi \|_{X^1}^2$ satisfies $\frac{d}{dt}X_k(t) \leq C_k(M_0^2+M_0^4)X_k(t)$. An application of Gronwall's inequality completes the proof of \eqref{est:aprioriZk}.
\end{proof}


\begin{proposition}[$\mathcal{Z}^1$ difference estimate]\label{prop:apriori-d1}
Let $\delta\geq 1$, $M_0>0$, and let $\phi, \psi \in \mathcal{Z}^2$ satisfy $\mathcal{E}^2_\delta(\phi) \leq M_0^2$ and $\mathcal{E}^2_\delta(\psi)\leq M_0^2$. 
 For $\ell,m\geq 0$, let $u_\ell,v_m\in C([0,T_0];\mathcal{Z}^\infty)$ be the unique solutions to the mollified problem \eqref{gINLSmoll} with initial data $P_{\leq \ell}\phi$ and $P_{\leq m}\psi$, respectively.
Then, there exists $T_*\in (0,T_0]$ with $T_*\gtrsim (1+M_0^4)^{-1}$ such that we have
\begin{equation}\label{est:apriorid1}
\| d^1_\delta(u_{\ell},v_{m})\|_{L_{T_*}^{\infty}} \lesssim (1+M_0)d^1_\delta(\phi,\psi) +(1+M_0^2)(2^{-\frac{1}{2}\ell} +2^{-\frac{1}{2}m}).
\end{equation}
In particular, if $\phi=\psi$, then  
\begin{equation}\label{est:apriorid1'}
\| d^1_\delta(u_{\ell}, u_{m}) \|_{L_{T_*}^{\infty}} \lesssim (1+M_0^2)(2^{-\frac{1}{2}\ell} +2^{-\frac{1}{2}m}). 
\end{equation}
\end{proposition}

\begin{proof}
Take an arbitrary $T\in (0,T_0]$ such that $T(1+M_0^4)\ll 1$. Recall from Proposition~\ref{prop:apriori-Zk} and the properties of the Littlewood-Paley projectors, we have the uniform a priori bounds
\begin{equation}\label{est:apriorid1-1}
\| \mathcal{E}^2_\delta(P_{\leq \ell}u_\ell)\|_{L^\infty_T}+\| \mathcal{E}^2_\delta(P_{\leq m}v_m)\|_{L^\infty_T} \lesssim M_0^2. 
\end{equation}

We begin by controlling the $d^1_\delta$ metric for the lower frequencies. 
By the definition in \eqref{def:d1}, we must estimate $\|u_\ell - v_m\|_{X^1}$ and $\delta^{-4/3}\widetilde{d}(u_\ell, v_m)$. 
We recall from \eqref{gINLSmoll-int} and \eqref{def:N0N1},
\[ u_\ell(t) = U(t)P_{\leq \ell}\phi - i \int_0^t U(t-s) P_{\leq \ell} \mathcal{N}_{\delta}(P_{\leq \ell}u_\ell(s)) \, ds - i \int_0^t U(t-s) P_{\leq \ell} \mathcal{N}_{\infty}(P_{\leq \ell}u_\ell(s)) \, ds. \]
Hence, we have
\begin{align*}
\| u_\ell -v_m\|_{X^1}&\lesssim \| U(t)\{ P_{\leq \ell}\phi -P_{\leq m}\psi \}\|_{X^1}+\Big\| \int_0^tU(t-s)\{ P_{\leq\ell}\mathcal{N}_\delta(P_{\leq \ell}u_\ell(s))-P_{\leq m}\mathcal{N}_\delta(P_{\leq m}v_m(s))\} \,ds\Big\|_{X^1} \\
&\quad +\| \partial_x(u_\ell-v_m)\|_{L^2}+\Big\| \int_0^tU(t-s)\{ P_{\leq\ell}\mathcal{N}_\infty(P_{\leq \ell}u_\ell(s))-P_{\leq m}\mathcal{N}_\infty(P_{\leq m}v_m(s))\} \,ds\Big\|_{L^2},
\end{align*}
where we have used the embedding $H^1(\mathbb{R})\hookrightarrow X^1$ and the integral equation to get rid of the $\dot{H}^1$-norm for the last term.
Invoking the linear bound \eqref{est:UtEk}, we obtain
\begin{equation}\label{est:d1-X1}
\begin{aligned}
\| u_\ell -v_m\|_{L^\infty_TX^1}&\lesssim \| P_{\leq \ell}\phi -P_{\leq m}\psi \|_{X^1}+T\| P_{\leq\ell}\mathcal{N}_\delta(P_{\leq \ell}u_\ell)-P_{\leq m}\mathcal{N}_\delta(P_{\leq m}v_m)\|_{L^\infty_TX^1}\\
&\quad +\| \partial_x(u_\ell-v_m)\|_{L^\infty_TL^2}+T\| P_{\leq\ell}\mathcal{N}_\infty(P_{\leq \ell}u_\ell)-P_{\leq m}\mathcal{N}_\infty(P_{\leq m}v_m)\|_{L^\infty_TL^2}.
\end{aligned}
\end{equation}
A similar argument with the metric inequality \eqref{est:d1f+g} implies
\begin{equation}\label{est:d1-d}
\begin{aligned}
\|\widetilde{d}(u_\ell,v_m)\|_{L^\infty_T}&\lesssim \widetilde{d}(P_{\leq \ell}\phi,P_{\leq m}\psi)+\big( 1+T\| \mathcal{N}(P_{\leq \ell}u_\ell)\|_{L^\infty_TL^2}\big) \| P_{\leq \ell}\phi -P_{\leq m}\psi\|_{X^1}\\
&\quad +\big( \| P_{\leq m}\psi\|_{X^1}+T\| \mathcal{N}(P_{\leq \ell}u_\ell)\|_{L^\infty_TX^1}+T\| \mathcal{N}(P_{\leq m}v_m)\|_{L^\infty_TX^1}\big) \\
&\qquad\qquad \times T\| P_{\leq\ell}\mathcal{N}(P_{\leq \ell}u_\ell)-P_{\leq m}\mathcal{N}(P_{\leq m}v_m)\|_{L^\infty_TL^2}.
\end{aligned}
\end{equation}
We now evaluate the nonlinear terms using the bounds from Lemma~\ref{lem:N0N1}, the uniform bounds \eqref{est:apriorid1-1} and the tail estimates such as $\| P_{\leq \ell}f-f\|_{L^2}\lesssim 2^{-\ell}\| \partial_xf\|_{L^2}$.
For instance, the following bounds are available:
\begin{align*}
&\| \mathcal{N}(P_{\leq \ell} u_\ell)\|_{L^\infty_TX^1}+\| \mathcal{N}(P_{\leq m}v_m)\|_{L^\infty_TX^1}\lesssim M_0^3,\\
&\| P_{\leq\ell}\mathcal{N}(P_{\leq \ell}u_\ell)-P_{\leq m}\mathcal{N}(P_{\leq m}v_m)\|_{L^\infty_TL^2}\\
&\quad \lesssim M_0^3(2^{-\ell}+2^{-m})+\delta^{\frac13}\big( M_0^2\| P_{\leq \ell}u_\ell-P_{\leq m}v_m\|_{L^\infty_TX^1}+M_0\| \delta^{-\frac43} \widetilde{d}(P_{\leq \ell}u_\ell ,P_{\leq m}v_m)\|_{L^\infty_T}\big) .
\end{align*}
Plugging these bounds into \eqref{est:d1-X1}--\eqref{est:d1-d} and invoking \eqref{est:d1PlPm}, we obtain
\begin{align*}
\| d^1_\delta(u_\ell,v_m)\|_{L^\infty_T} &\lesssim \| \partial_x(u_\ell-v_m)\|_{L^\infty_TL^2} +(1+M_0+TM_0^3)\| \phi-\psi\|_{X^1}+\delta^{-\frac43}\widetilde{d}(\phi,\psi) \\
&\quad +(1+M_0+TM_0^3)(M_0+TM_0^3)(2^{-\ell}+2^{-m}) \\
&\quad +(1+M_0+TM_0^3)T(M_0+M_0^2) \| d^1_\delta(u_\ell,v_m)\|_{L^\infty_T}.
\end{align*}
Imposing smallness on the time step (here, $T(1+M_0^3)\ll 1$ is sufficient), we can absorb the last term into the left-hand side, leaving us with
\begin{equation}\label{est:apriorid1-A}
\| d^1_\delta(u_\ell,v_m)\|_{L^\infty_T} \lesssim (1+M_0)d^1_\delta(\phi,\psi)+(2^{-\ell}+2^{-m})(1+M_0^2) +\| \partial_x(u_\ell-v_m)\|_{L^\infty_TL^2}.
\end{equation}

It remains to estimate the last term in \eqref{est:apriorid1-A}. Differentiating the $L^2$ norm of the difference and applying the equation \eqref{gINLSmoll}, we have
\[
\frac{d}{dt}\big\| \partial_x\big( u_\ell(t)-v_m(t)\big) \big\|_{L^2}^2 = 2\Re \big( {-}i\partial_x\big( P_{\leq \ell}\mathcal{N}(P_{\leq \ell}u_\ell(t))-P_{\leq m}\mathcal{N}(P_{\leq m}v_m(t))\big) , \partial_x\big( u_\ell(t)-v_m(t)\big) \big)_{L^2}.
\]
By using Bernstein's inequality, integration by parts, Lemma~\ref{lem:N0N1}, and the bounds in \eqref{est:apriorid1-1}, we can shift the frequency projections to act entirely on the difference, generating a harmless commutator remainder
\begin{align*}
    \frac{d}{dt}\big\| \partial_x\big( u_\ell(t)-v_m(t)\big) \big\|_{L^2}^2 
    &= 2\Re \big( {-}i\partial_x\big( \mathcal{N}(P_{\leq \ell}u_\ell(t))-\mathcal{N}(P_{\leq m}v_m(t))\big) , \partial_x\big( P_{\leq \ell}u_\ell(t)-P_{\leq m}v_m(t)\big) \big)_{L^2} \\
    & +\mathcal{R}_1(t)
\end{align*} 
where the remainder satisfies $|\mathcal{R}_1(t)|\lesssim (2^{-\ell}+2^{-m})M_0^4$. To treat the nonlinear difference, we expand it algebraically to isolate the problematic derivatives
\begin{align*}
    \mathcal{N}(P_{\leq \ell}u_\ell)-\mathcal{N}(P_{\leq m}v_m)
    &=\mathcal{N}_\delta(P_{\leq \ell}u_\ell)-\mathcal{N}_\delta(P_{\leq m}v_m)  \\
    &\quad +\frac{i}2 (P_{\leq \ell}u_\ell-P_{\leq m}v_m) (\alpha +i\beta H)\partial_x (|P_{\leq \ell}u_\ell|^2+|P_{\leq m}v_m|^2) \\
    &\quad +i\tfrac{P_{\leq \ell}u_\ell+P_{\leq m}v_m}{2}(\alpha +i\beta H)\partial_x\left(\overline{\tfrac{P_{\leq \ell}u_\ell+P_{\leq m}v_m}{2}}(P_{\leq \ell}u_\ell-P_{\leq m}v_m)\right) \\
    &\quad+i\tfrac{P_{\leq \ell}u_\ell+P_{\leq m}v_m}{2}(\alpha +i\beta H)\partial_x \left(\tfrac{P_{\leq \ell}u_\ell+P_{\leq m}v_m}{2}\overline{(P_{\leq \ell}u_\ell-P_{\leq m}v_m)}\right) ,
\end{align*} 
where $\mathcal{N}_\delta(\cdot)$ is defined in \eqref{def:N0N1}. Applying Lemma~\ref{lem:N0N1} and Lemma~\ref{lem:ibp}, we control
\[ \| \partial_x(\mathcal{N}_\delta(P_{\leq \ell}u_\ell)-\mathcal{N}_\delta(P_{\leq m}v_m))\|_{L^\infty_TL^2}\| \partial_x(P_{\leq \ell}u_\ell-P_{\leq m}v_m)\|_{L^\infty_TL^2} \lesssim (M_0+M_0^2)\|  d^1_\delta(P_{\leq \ell}u_\ell,P_{\leq m}v_m)\|_{L^\infty_T}^2 .\]
 For the remaining terms, we use integration by parts to distribute derivatives. Most interactions are bounded by $M_0^2\| P_{\leq \ell}u_\ell-P_{\leq m}v_m\|_{X^1}^2$ thanks to Lemma~\ref{lem:ibp} and \eqref{est:apriorid1-1}. The only term that cannot be bounded directly is the one involving $\beta H$ and two spatial derivatives on the difference. Isolating this interaction, we obtain 
\begin{equation}\label{est:apriorid1-2}
\begin{aligned}
\frac{d}{dt}\big\| \partial_x\big( u_\ell(t)-v_m(t)\big) \big\|_{L^2}^2
&\leq 2\beta \Re \Big( \tfrac{P_{\leq \ell}u_\ell+P_{\leq m}v_m}{2}(iH)\Big(\tfrac{P_{\leq \ell}u_\ell+P_{\leq m}v_m}{2}\partial_x^2\overline{(P_{\leq \ell}u_\ell-P_{\leq m}v_m)}\Big) , \\
&\hspace{2cm} \partial_x\big(P_{\leq \ell}u_\ell-P_{\leq m}v_m\big) \Big)_{L^2} \\
&\quad +C(2^{-\ell}+2^{-m})M_0^4+C(1+M_0^2)\| d^1_\delta(P_{\leq \ell}u_\ell,P_{\leq m}v_m)\|_{L^\infty_T}^2.
\end{aligned}
\end{equation}

Since the first term on the right-hand side of \eqref{est:apriorid1-2} prevents closing the estimate, we introduce a time-dependent corrector $\mathcal{I}(t)$ for a frequency threshold $n_2\in \Z_{\geq 0}$
\[ 
\mathcal{I}(t):=\beta \Re \Big( \tfrac{P_{\leq \ell}u_\ell+P_{\leq m}v_m}{2}H\partial_x \Big(\tfrac{P_{\leq \ell}u_\ell+P_{\leq m}v_m}{2}\overline{P_{\geq n_2}\big(P_{\leq \ell}u_\ell-P_{\leq m}v_m\big)}\Big) , P_{\geq n_2}\big(P_{\leq \ell}u_\ell-P_{\leq m}v_m\big) \Big)_{L^2} .
\]
Differentiating $\mathcal{I}(t)$
\begin{equation}\label{est:apriorid1-3}
\begin{aligned}
    \frac{d}{dt}\mathcal{I}(t)
    &= -2\beta \Re \bigg( iH\Big(\tfrac{P_{\leq \ell}u_\ell+P_{\leq m}v_m}{2} \partial_x^2\overline{(P_{\leq \ell}u_\ell - P_{\leq m}v_m)}\Big) , \overline{\tfrac{P_{\leq \ell}u_\ell+P_{\leq m}v_m}{2}}\partial_x(P_{\leq \ell}u_\ell - P_{\leq m}v_m)\bigg)_{L^2} \\
    &\quad + \mathcal{R}_2(t),
\end{aligned}
\end{equation}
where 
\[ |\mathcal{R}_2(t)| \lesssim (2^{-\ell}+2^{-m})M_0^6 + \big( (1+2^{n_2})M_0^2+M_0^4\big) \| d^1_\delta(P_{\leq \ell}u_\ell,P_{\leq m}v_m)\|_{L^\infty_T}^2. \] 
Combining this with \eqref{est:apriorid1-2} yields
\begin{equation}\label{est:apriorid1-5}
\begin{aligned}
\frac{d}{dt}\Big( \big\| \partial_x\big( u_\ell(t)-v_m(t)\big) \big\|_{L^2}^2 +\mathcal{I}(t)\Big) &\lesssim (2^{-\ell}+2^{-m})(M_0^4+M_0^6) \\
&\quad +(1+2^{n_2}M_0^2+M_0^4) \| d^1_\delta(P_{\leq \ell}u_\ell,P_{\leq m}v_m)\|_{L^\infty_T}^2.
\end{aligned} 
\end{equation}
Integrating \eqref{est:apriorid1-5} over $[0, t]$ and utilizing the bound $\| \mathcal{I}(t)\|_{L^\infty_T}\lesssim 2^{-n_2}M_0^2\| P_{\leq \ell}u_\ell-P_{\leq m}v_m\|_{L^\infty_TX^1}^2$, we deduce
\begin{align}\label{est:apriorid1-B}
\big\| \partial_x\big( u_\ell-v_m\big) \big\|_{L^\infty_TL^2}^2
&\lesssim \Big( \|\partial_x(\phi-\psi)\|_{L^2}+(2^{-\ell}+2^{-m})M_0\Big) ^2 + (2^{-\ell}+2^{-m})T(M_0^4+M_0^6) \\
&\quad +\Big\{ T(1+2^{n_2}M_0^2+M_0^4) + 2^{-n_2}M_0^2\Big\} \Big( \| d^1_\delta(u_\ell,v_m)\|_{L^\infty_T}+(2^{-\ell}+2^{-m})(M_0+M_0^2)\Big)^2.\notag
\end{align}
Substituting this into \eqref{est:apriorid1-A}, we set $2^{n_2}=C(1+M_0^2)$ and $T \ll (1+M_0^4)^{-1}$. 
Absorbing the difference term into the left-hand side establishes the claimed bound \eqref{est:apriorid1}.
\end{proof}


\medskip
At this stage, we are now ready to complete the proof of Theorem~\ref{thm:LWP}.

\begin{proof}[Proof of Theorem \ref{thm:LWP}]
Let $\phi\in \mathcal{Z}^2$ satisfy $\mathcal{E}^2_\delta(\phi)\leq M_0^2$, and let $T_*>0$ be the constant given in Proposition~\ref{prop:apriori-d1}.
We divide the proof into several steps.

\medskip
\noindent \textbf{Step 1.} 
\textit{We first construct a local solution $u\in C([0,T_*];\mathcal{Z}^1)\cap L^\infty(0,T_*;\mathcal{Z}^2)$ to \eqref{gINLS} with initial data $\phi$ such that $u\in C([0,T_*];\mathcal{Z}^1_{\rho_\phi})$ and
\begin{align}
\big\| \mathcal{E}^2_\delta(u)\big\|_{L^\infty_{T_*}}&\lesssim \mathcal{E}^2_\delta(\phi), \label{bd:E2} \\
\| \widetilde{\mathcal{E}}(u)\|_{L^\infty_{T_*}}&\lesssim \widetilde{\mathcal{E}}(\phi) +\| \phi\|_{X^1}^2. \label{bd:Eti}
\end{align}
We also show that if $\phi\in \mathcal{Z}^k$ for some integer $k\geq 3$, then $u\in L^\infty(0,T_*;\mathcal{Z}^k)$ and satisfies
\begin{equation}\label{bd:Ek}
\| u\|_{L^\infty_{T_*}X^k}\lesssim_k \| \phi\|_{X^k}.
\end{equation}}

Let $\{ u_\ell\}_{\ell\geq 0}$ be the sequence of solutions to the mollified problem \eqref{gINLSmoll} with initial data $u_\ell(0)=P_{\leq \ell}\phi$, constructed in Proposition~\ref{prop:LWPmoll}. Propositions~\ref{prop:apriori-Zk} and~\ref{prop:apriori-d1} guarantee that \(u_\ell\in C([0,T_*];\mathcal{Z}^2)\) for all \(\ell\geq 0\) and that \(\{ u_\ell\}_{\ell\geq 0}\) is Cauchy in the complete metric space \(C([0,T_*];\mathcal{Z}^1)\).
Hence, there exists \(u\in C([0,T_*];\mathcal{Z}^1)\) such that \(u_\ell\to u\) strongly in \(C([0,T_*];\mathcal{Z}^1)\).
Moreover, \eqref{est:aprioriZ2} shows that the sequence \(\{ \partial_xu_\ell\}_{\ell\geq 0}\) is bounded in \(L^\infty(0,T_*;H^1(\R))\). By the Banach-Alaoglu theorem, we can extract a subsequence converging weakly-$*$ to some $\widetilde{u}\in L^\infty(0,T_*;H^1(\R))$. Since \(\partial_xu_\ell\to \partial_xu\) strongly in \(C([0,T_*];L^2(\R))\), we have \(\partial_xu=\widetilde{u}\in L^\infty(0,T_*;H^1(\R))\); thus \(u\in L^\infty(0,T_*;\mathcal{Z}^2)\). The bound \eqref{bd:E2} follows immediately from \eqref{est:aprioriZ2} and the weak-$*$ lower semicontinuity of the $L^\infty_{T_*}H^1$ norm, {while \eqref{bd:Eti} is a consequence of \eqref{est:aprioriZ2'} and the strong convergence in $C([0,T_*];\mathcal{Z}^1)$}. The higher regularity \eqref{bd:Ek} for \(k\geq 3\) is obtained analogously using \eqref{est:aprioriZ2'} and \eqref{est:aprioriZk}.

\medskip
It remains to verify that the limit \(u\) satisfies the integral equation
\begin{equation}\label{gINLS-int}
u(t)=U(t)\phi -i\int_0^tU(t-s)\mathcal{N}(u(s))\,ds,
\end{equation}
that \(u(0)=\phi\)  and that \(u\in C([0,T_*];\mathcal{Z}^1_{\rho_\phi})\).
The initial condition $u(0)=\phi$ is clear from the strong convergence.
To pass to the limit in the nonlinearity, we estimate
\begin{align*}
&\| P_{\leq \ell}\mathcal{N}(P_{\leq \ell}u_\ell)-\mathcal{N}(u)\|_{L^\infty_{T_*}L^2}\\
&\lesssim 2^{-\ell}\| \mathcal{N}(P_{\leq \ell}u_\ell)\|_{L^\infty_{T_*}H^1}+\| \mathcal{N}(P_{\leq \ell}u_\ell)-\mathcal{N}(u)\|_{L^\infty_{T_*}L^2} \\
&\lesssim 2^{-\ell}\delta^{\frac13}\| \mathcal{E}^2_\delta(P_{\leq \ell}u_\ell)\|_{L^\infty_{T_*}}^{\frac32}+\delta^{\frac13}\Big( 1+\| \mathcal{E}^1_\delta(P_{\leq \ell}u_\ell)\|_{L^\infty_{T_*}}+\| \mathcal{E}^1_\delta(u)\|_{L^\infty_{T_*}}\Big) \| d^1_\delta(P_{\leq \ell}u_\ell,u)\|_{L^\infty_{T_*}} \\
&\lesssim_{\delta,M_0} 2^{-\ell}+\| d^1_\delta(u_\ell,u)\|_{L^\infty_{T_*}} \\
&\to \ 0\qquad (\text{as } \ell\to \infty).
\end{align*}
Here, we utilized Lemma~\ref{lem:N0N1}, \eqref{est:dconv}, \eqref{est:aprioriZ2}, \eqref{bd:E2}, and the strong convergence $u_\ell\to u$ in $C([0,T_*];\mathcal{Z}^1)$. Consequently,
\[ -i\int_0^t U(t-s)P_{\leq \ell}\mathcal{N}(P_{\leq \ell}u_\ell(s))\,ds\ \to \ -i\int_0^t U(t-s)\mathcal{N}(u(s))\,ds \quad\text{strongly in } C([0,T_*];L^2(\R)). \]
Together with the convergence of the linear part, this shows that \(u\) satisfies \eqref{gINLS-int}.
Because \(U(t)\phi\in C(\R;\mathcal{Z}^1_{\rho_\phi})\) by \eqref{est:UtEk} and the Duhamel integral belongs to \(C([0,T_*];H^1(\R))\subset C([0,T_*];\mathcal{Z}^1_0)\) for \(u\in L^\infty(0,T_*;\mathcal{Z}^2)\) by Lemma~\ref{lem:N0N1}, we conclude from \eqref{gINLS-int} that \(u\in C([0,T_*];\mathcal{Z}^1_{\rho_\phi})\).

\medskip
Finally, let \(\phi,\psi\in \mathcal{Z}^2\) with \(\mathcal{E}^2_\delta(\phi), \mathcal{E}^2_\delta(\psi)\leq M_0^2\) and let \(u,v\) be the corresponding solutions constructed above.
Taking \(\ell,m\to \infty\) in \eqref{est:apriorid1} yields 
\begin{equation}\label{bd:d1}
\big\| d^1_\delta(u,v)\big\|_{L^\infty_{T_*}}\lesssim (1+M_0)\,d^1_\delta(\phi,\psi).
\end{equation}

\bigskip
\noindent
{\bf Step 2}.
{\it We claim the following: 
Let $u,v\in C([0,T];\mathcal{Z}^2)$ be two solutions to \eqref{gINLS} on an interval $[0,T]$. 
Assume that $\| \mathcal{E}^2_\delta(u)\|_{L^\infty_T}\leq M_1^2$ and $\| \mathcal{E}^2_\delta(v)\|_{L^\infty_T}\leq M_1^2$ for some $M_1>0$.
Then, there exists $T_1=T_1(M_1)>0$ such that for $0\leq t\leq \min \{ T,T_1\}$ it holds that
\begin{equation*}
d^1_\delta(u(t),v(t))\lesssim (1+M_1)d^1_\delta(u(0),v(0)).
\end{equation*}
In particular, the solution of the Cauchy problem associated with \eqref{gINLS} is unique in $C([0,T];\mathcal{Z}^2)$.}

\medskip
We remark that this claim must be discussed independently of the previous Step 1 (in particular the difference estimate \eqref{bd:d1}), since for uniqueness we have to consider general solutions which may differ from those constructed in Step 1.
However, the claim can be proven along the same lines as Proposition~\ref{prop:apriori-d1} using the original equation \eqref{gINLS}, \eqref{gINLS-int} instead of the mollified one \eqref{gINLSmoll}, \eqref{gINLSmoll-int}.
We omit the details.

\bigskip
\noindent
{\bf Step 3}.
{\it We shall prove that the solution $u$ constructed in Step~1 satisfies $u \in C([0,T_*];\mathcal{Z}^2)$ and the continuity of the flow map $u_0\mapsto u$ in $C([0,T_*];\mathcal{Z}^2)$.}

\medskip
We can prove $u\in C([0,T_*];\mathcal{Z}^k)$ for $\phi\in \mathcal{Z}^k$ if $k\geq 3$ in a similar manner; we omit the proof.
In particular, this concludes the proof of Theorem~\ref{thm:LWP}.
To prove the claim, we may adapt an argument due to Bona--Smith \cite{Bona-Smith} or the idea of frequency envelope due to Tao \cite{Tao-2001, Ifrim-Tataru, Burq} to the nonvanishing solutions in $\mathcal{Z}^2$.%
\footnote{~The abstract theorem in \cite{Burq} does not seem to apply directly to our problem, since $\mathcal{Z}^2$ is not a vector space.}
Here, we outline an adaptation of the latter idea.

Fix $\eps \in (0,1)$, and for $\delta>0$ define the frequency envelope $\mathbf{c}_\delta [f]=\{ c_{\delta,j}[f]\}_{j\geq 0}$ of $f\in \mathcal{Z}^2$ by%
\footnote{~We also define $\mathbf{c}_\infty [f]$ for $f\in X^2$ by dropping the $\delta$-depending term in $\mathbf{c}_\delta [f]$.
We can easily show the properties of $\mathbf{c}_\infty$ similar to those of $\mathbf{c}_\delta$ in Lemma~\ref{lem:fe}.}
\[ c_{\delta,j}[f]:=\sum _{k=0}^{\infty}2^{-\eps |j-k|}\| P_k\partial_xf\|_{H^1}+2^{-\eps j}\big( \| f\|_{L^\infty}+\delta^{-\frac23}\big\| |f|^2-\rho_f^2\big\|_{L^2}^{\frac12}\big) ,\qquad j\in\Z_{\geq 0}.\]
We collect some basic properties of $\mathbf{c}_\delta [f]$ in the following lemma, which can be shown by a standard argument (see, e.g., \cite{Ifrim-Tataru}); we give a proof in Appendix~\ref{sec:app-fe} for completeness.
\begin{lemma}\label{lem:fe}
Let $\delta>0$.
For any $f,g\in \mathcal{Z}^2$ and integers $j,k\geq 0$, the following hold.
\begin{enumerate}
\item[$\mathrm{(i)}$] $\| P_j\partial_xf\|_{H^1}\leq c_{\delta,j}[f]$,\quad $c_{\delta,k}[f]\leq 2^{\eps |k-j|}c_{\delta,j}[f]$,\quad $\mathcal{E}^2_\delta(f)\sim \| \mathbf{c}_\delta[f]\|_{\ell^2(\Z_{\geq 0})}^2$.
\item[$\mathrm{(ii)}$] $\big\| \mathbf{c}_\delta[f]-\mathbf{c}_\delta[g]\big\| _{\ell^2(\Z_{\geq 0})}\lesssim \| f-g\|_{X^2}+\delta^{-\frac23}\big\| (|f|^2-\rho_f^2)-(|g|^2-\rho_g^2)\big\|_{L^2}^{1/2}$.
\item[$\mathrm{(iii)}$] $\| P_{\leq j}f\|_{X^3}\lesssim 2^jc_{\delta,j}[f]$.
\end{enumerate}
\end{lemma}

For any $\delta\geq 1$, $\phi\in \mathcal{Z}^2$ with $\mathcal{E}^2_\delta(\phi)\leq M_0^2$ and $j\geq 1$, let $u$ and $u^j$ be the unique solutions of \eqref{gINLS} with initial condition, respectively, $u(0)=\phi$ and $u^j(0)=P_{\leq j}\phi$, which we have constructed in the class $C([0,T_*];\mathcal{Z}^1)\cap L^\infty(0,T_*;\mathcal{Z}^2)$.%
\footnote{~From \eqref{est:EkPl}, we see that $\mathcal{E}^2_\delta(P_{\leq j}\phi)\leq 3M_0^2$ for any $j\geq 0$.
Therefore, by slightly modifying the choice of $T_*$ in Proposition~\ref{prop:apriori-d1} we can construct the solutions $u$, $u^j$ on a uniform interval $[0,T_*]$.}
Note that $P_{\leq j}\phi\in \mathcal{Z}^3$ and hence $u^j\in C([0,T_*];\mathcal{Z}^1)\cap L^\infty(0,T_*;\mathcal{Z}^3)\subset C([0,T_*];\mathcal{Z}^2)$.
Using the $X^3$ bound from \eqref{bd:Ek}, the $\mathcal{Z}^1$ difference bound \eqref{bd:d1}, and Lemma~\ref{lem:fe}, we can derive the following estimate:
\begin{equation}\label{est:fe}
\big\| P_k\partial_x^2(u^j-u^{j-1})\big\|_{L^\infty_{T_*}L^2}\lesssim_{M_0} 2^{-|k-j|}c_{\delta,j}[\phi],\qquad j\geq 1,\quad k\geq 0.
\end{equation}
This estimate is the key to prove the claim of Step~3 and shows that the initial frequency distribution is not significantly disturbed during the time of local existence. 
For the reader's convenience, we also give a proof of \eqref{est:fe} in Appendix~\ref{sec:app-fe}.

To see $u\in C([0,T_*];\mathcal{Z}^2)$, we first notice that the difference bound \eqref{bd:d1} shows the convergence $u^j\to u$ in $C([0,T_*];\mathcal{Z}^1)$, and hence $u\in C([0,T_*];\mathcal{Z}^1)$.
It then suffices to prove $u\in C([0,T_*];\dot{H}^2(\R))$.
We deduce from \eqref{est:fe} that for any $m_2>m_1\geq 0$,
\begin{align*}
\| \partial_x^2(u^{m_2}-u^{m_1})\|_{L^\infty_{T_*}L^2}^2&=\bigg\| \sum_{j=m_1+1}^{m_2}\partial_x^2(u^j-u^{j-1})\bigg\|_{L^\infty_{T_*}L^2}^2 \\
&\lesssim \sum_{j_1,j_2=m_1+1}^{m_2}\sum_{\begin{smallmatrix} k_1,k_2\geq 0 \\ |k_1-k_2|\leq 1\end{smallmatrix}}\| P_{k_1}\partial_x^2(u^{j_1}-u^{j_1-1})\|_{L^\infty_{T_*}L^2}\| P_{k_2}\partial_x^2(u^{j_2}-u^{j_2-1})\|_{L^\infty_{T_*}L^2} \\
&\lesssim_{M_0} \sum_{j_1,j_2=m_1+1}^{m_2}\sum_{\begin{smallmatrix} k_1,k_2\geq 0 \\ |k_1-k_2|\leq 1\end{smallmatrix}}2^{-|k_1-j_1|}c_{\delta,j_1}[\phi] 2^{-|k_2-j_2|}c_{\delta,j_2}[\phi] \\
&\lesssim \sum_{j_1,j_2=m_1+1}^{m_2} 2^{-\frac12|j_1-j_2|}c_{\delta,j_1}[\phi]c_{\delta,j_2}[\phi] \lesssim \sum_{j=m_1+1}^{m_2}c_{\delta,j}[\phi]^2
\leq \| \mathbf{c}_\delta[\phi]\|_{\ell^2(\mathbb{Z}_{>m_1})}^2.
\end{align*}
Since $\mathbf{c}_\delta[\phi]\in \ell^2(\Z_{\geq 0})$, we see that the sequence $\{ u^j\}_{j\geq 0}$ is Cauchy in $C([0,T_*];\dot{H}^2(\R))$, concluding that the limit $u$ belongs to $C([0,T_*];\mathcal{Z}^2)$ and that
\begin{equation}\label{est:cont1}
\| \partial_x^2(u-u^j)\|_{L^\infty_{T_*}L^2}\lesssim_{M_0} \| \mathbf{c}_\delta[\phi]\|_{\ell^2(\Z_{>j})},\qquad j\geq 0.
\end{equation}

To see the continuity of the flow map, let $\phi, \{ \phi_{(n)} \}_{n\geq 0}$ be such that $\mathcal{E}^2_\delta(\phi), \mathcal{E}^2_\delta(\phi_{(n)})\leq M_0^2$ and $d^2_\delta(\phi_{(n)},\phi)\to 0$ as $n\to \infty$. 
Furthermore, let $u_{(n)}, u_{(n)}^{j} \in C([0,T_*];\mathcal{Z}^2)$ denote the solutions to \eqref{gINLS} with $u_{(n)}(0)=\phi_{(n)}$ and $u_{(n)}^{j}(0)=P_{\leq j}\phi_{(n)}$. 
In view of the $\mathcal{Z}^1$ difference estimate \eqref{bd:d1}, it suffices to prove
\begin{equation}\label{est:cont2}
\lim_{n\to \infty}\| \partial_x^2(u_{(n)} - u) \|_{L_{T_*}^{\infty}L^2}=0. 
\end{equation}
The triangle inequality yields
\[ \| \partial_x^2(u_{(n)} - u) \|_{L_{T_*}^{\infty}L^2}\leq \| \partial_x^2(u_{(n)}-u_{(n)}^j)\|_{L^\infty_{T_*}L^2}+\| \partial_x^2(u-u^j)\|_{L^\infty_{T_*}L^2}+\| \partial_x^2(u_{(n)}^j-u^j)\|_{L^\infty_{T_*}L^2} .\]
Using \eqref{est:cont1} and Lemma~\ref{lem:fe} (ii), we bound the first two terms on the right-hand side by
\[ \| \mathbf{c}_\delta[\phi_{(n)}]\|_{\ell^2(\Z_{>j})}+\| \mathbf{c}_\delta[\phi]\|_{\ell^2(\Z_{>j})} \lesssim \| \phi_{(n)}-\phi\|_{X^2}+\big[ d^1_\delta(\phi_{(n)},\phi)\big]^{\frac12}+\| \mathbf{c}_\delta[\phi]\|_{\ell^2(\Z_{>j})}.\]
For the last term, we can use the $X^3$ bound from \eqref{bd:Ek} and the difference bound \eqref{bd:d1} as well as \eqref{est:dkPl}:
\begin{equation}\label{est:cont3}
\begin{aligned}
\| \partial_x^2(u_{(n)}^j-u^j)\|_{L^\infty_{T_*}L^2}^2&\leq \big( \| \partial_x^3u_{(n)}^j\|_{L^\infty_{T_*}L^2}+\| \partial_x^3u^j\|_{L^\infty_{T_*}L^2}\big) \| d^1_\delta(u_{(n)}^j,u^j)\|_{L^\infty_{T_*}} \\
&\lesssim_{M_0} 2^j\, d^1_\delta(P_{\leq j}\phi_{(n)},P_{\leq j}\phi) 
\lesssim_{M_0} 2^j\,d^1_\delta(\phi_{(n)},\phi).
\end{aligned}
\end{equation} 
Putting these estimates together, we have
\[ \limsup_{n\to \infty}\| \partial_x^2(u_{(n)}-u)\|_{L^\infty_{T_*}L^2}\lesssim_{M_0} \| \mathbf{c}_\delta[\phi]\|_{\ell^2(\Z_{>j})} \]
for any $j$, and obtain \eqref{est:cont2} by letting $j\to \infty$.

We have completed the proof of Theorem~\ref{thm:LWP}.
\end{proof}

\begin{remark}
Corollary~\ref{cor:LWP-CM} follows from the exact same argument by simply omitting the nonlinear term $\mathcal{N}_\delta$. 
We first disregard the estimates of $\widetilde{\mathcal{E}}(u)$ and $\widetilde{d}(u,v)$ to prove local well-posedness in $X^2$, and then employ these estimates to prove persistence and continuity of the solution map in $\mathcal{Z}^2$.
\end{remark}

We conclude this section by showing the convergence result.
\begin{proof}[Proof of Theorem~\ref{thm:conv}]
Let $\{ \phi_\delta\}_{\delta\in [1,\infty]}$, $\{ u_\delta\}_{\delta\in[1,\infty]}$ be as in the statement of the theorem, and take $M_0>0$ such that $\| \phi_\infty\|_{X^2}^2\leq M_0^2$. 
Then, by the assumption and \eqref{est:EkPl}, there exists $\delta_*\in [1,\infty)$ such that $\mathcal{E}^2_\delta(\phi_\delta)\leq 2M_0^2$ and $\mathcal{E}^2_\delta(P_{\leq j}\phi_\delta)\leq 6M_0^2$ for any $\delta\in [\delta_*,\infty )$ and $j\in \Z_{\geq 0}$.
Let $u_\delta^j$ denote the (unique) solution to \eqref{gINLS} ($\delta\in [1,\infty)$) or \eqref{gCM} ($\delta=\infty$) with $u_\delta^j(0)=P_{\leq j}\phi_\delta$. 
From the proof of Theorem~\ref{thm:LWP}, there exists $T_*\gtrsim (1+M_0^4)^{-1}$ such that  for any $\delta\in [\delta_*,\infty )$ and $j\in \Z_{\geq 0}$, these solutions exist on $[0,T_*]$ (in particular $T_{\delta,\max}>T_*$) and satisfy
\begin{align}
\| \mathcal{E}^2_\delta(u_\delta)\|_{L^\infty_{T_*}}+\| \mathcal{E}^2_\delta(u_\delta^j)\|_{L^\infty_{T_*}}&\lesssim \mathcal{E}^2_\delta(\phi_\delta)+\mathcal{E}^2_\delta(P_{\leq j}\phi_\delta)\lesssim M_0^2, \label{est:apriori-delta1} \\
\| \widetilde{\mathcal{E}}(u_\delta)\|_{L^\infty_{T_*}}+\| \widetilde{\mathcal{E}}(u_\delta^j)\|_{L^\infty_{T_*}}&\lesssim \widetilde{\mathcal{E}}(\phi_\delta)+\widetilde{\mathcal{E}}(P_{\leq j}\phi_\delta)+\| \phi_\delta\|_{X^1}^2\lesssim \widetilde{\mathcal{E}}(\phi_\delta)+M_0^2, \label{est:apriori-delta1'} \\
\| u_\delta^j \|_{L^\infty_{T_*}X^3}&\lesssim \| P_{\leq j}\phi_\delta\|_{X^3}\lesssim 2^jM_0, \label{est:apriori-delta2} \\
\| d^1_\delta(u_\delta^j,u_\delta)\|_{L^\infty_{T_*}}&\lesssim_{M_0} d^1_\delta(P_{\leq j}\phi_\delta,\phi_\delta)\lesssim 2^{-j}(M_0+M_0^2), \label{est:apriori-delta3} \\
\| \partial_x^2(u_\delta-u_\delta^j)\|_{L^\infty_{T_*}L^2}&\lesssim _{M_0}\| \mathbf{c}_{\delta}[\phi_\delta]\|_{\ell^2(\Z_{>j})}. \label{est:apriori-delta4}
\end{align}
Indeed, these estimates have been shown in \eqref{bd:E2}, \eqref{bd:Eti}, \eqref{bd:Ek}, \eqref{bd:d1}, and \eqref{est:cont1}, respectively, with the constants independent of $\delta\in [1,\infty )$. 
Similarly, for $\delta=\infty$ and any $j\geq 0$ we have the corresponding estimates 
\begin{gather*}
\| u_\infty\|_{L^\infty_{T_*}X^2}+\| u_\infty^j\|_{L^\infty_{T_*}X^2}\lesssim M_0,\qquad \| u_\infty^j\|_{L^\infty_{T_*}X^3}\lesssim \| P_{\leq j}\phi_\infty\|_{X^3}\lesssim 2^jM_0, \\
\| u_\infty^j-u_\infty\|_{L^\infty_{T_*}X^1}\lesssim _{M_0}\| P_{\geq j+1}\phi_\infty\|_{X^1}\lesssim 2^{-j}M_0,\qquad \| \partial_x^2(u_\infty-u_\infty^j)\|_{L^\infty_{T_*}L^2}\lesssim _{M_0}\| \mathbf{c}_\infty[\phi_\infty]\|_{\ell^2(\Z_{>j})}.
\end{gather*}

We next prove the difference estimate
\begin{equation}\label{est:diff-delta}
\| u_\delta - u_\infty \|_{L^\infty_{T_*}X^1}\lesssim \| \phi_\delta-\phi_\infty\|_{X^1} + \nu_{M_0}(\delta),\qquad \delta\in [\delta_*,\infty),
\end{equation}
where $\nu_{M_0}(\delta)$ is a nonnegative function satisfying $\lim\limits_{\delta\to\infty}\nu_{M_0}(\delta)=0$ for each $M_0>0$.
To see \eqref{est:diff-delta}, it suffices to prove
\begin{equation}\label{est:diff-delta-j}
\| u_\delta^j-u_\infty^j\|_{L^\infty_{T_*}X^1}\lesssim \| \phi_\delta-\phi_\infty\|_{X^1} + \nu_{M_0}(\delta),\qquad \delta\in [\delta_*,\infty),\ j\geq 0
\end{equation}
and take the limit $j\to \infty$ with the estimate \eqref{est:apriori-delta3}.
Then, since $u_\delta^j$ is smooth, we can follow the calculations in the proof of Proposition~\ref{prop:apriori-d1}.
In fact, the main part of the proof involving the modified energy estimate remains unchanged; we only have to replace the difference of $\mathcal{N}_\delta$'s for two solutions by a single $\mathcal{N}_\delta(u_\delta^j)$.
The previous argument for Proposition~\ref{prop:apriori-d1} implies the estimates
\begin{align*}
\| u_\delta^j-u_\infty^j\|_{L^\infty_{T_*}X^1}&\lesssim \| P_{\leq j}(\phi_\delta-\phi_\infty)\|_{X^1}+T_*\| \mathcal{N}_\delta(u_\delta^j)\|_{L^\infty_{T_*}X^1}+\| \partial_x(u_\delta^j-u_\infty^j)\|_{L^\infty_{T_*}L^2}, \\
\| \partial_x( u_\delta^j-u_\infty^j) \|_{L^\infty_{T_*}L^2}^2 &\lesssim \| \partial_xP_{\leq j}(\phi_\delta-\phi_\infty)\|_{L^2}^2 + T_*(M_0+M_0^3)\| \mathcal{N}_\delta(u_\delta^j)\|_{L^\infty_{T_*}X^1}\\
&\quad +\big(T_*(2^{n_2}M_0^2+M_0^4) + 2^{-n_2}M_0^2\big) \| u_\delta^j-u_\infty^j\|_{L^\infty_{T_*}X^1}^2 
\end{align*}
instead of \eqref{est:apriorid1-A} and \eqref{est:apriorid1-B}, respectively.
Choosing $n_2$ appropriately and modifying $T_*\ll (1+M_0^4)^{-1}$ if necessary, we obtain
\[ \| u_\delta^j-u_\infty^j\|_{L^\infty_{T_*}X^1}^2\lesssim \| \phi_\delta-\phi_\infty\|_{X^1}^2+ (M_0+M_0^3)\| \mathcal{N}_\delta(u_\delta^j)\|_{L^\infty_{T_*}X^1}+\| \mathcal{N}_\delta(u_\delta^j)\|_{L^\infty_{T_*}X^1}^2.\]
Now, Lemma~\ref{lem:N0N1}, \eqref{est:apriori-delta1}, \eqref{est:apriori-delta1'}, and the assumption $\lim\limits_{\delta\to\infty}\delta^{-\frac43}\widetilde{\mathcal{E}}(\phi_\delta)=0$ imply that
\begin{align*}
\| \mathcal{N}_\delta(u_\delta^j)\|_{L^\infty_{T_*}X^1}&\lesssim \delta^{-1}\big( \| \widetilde{\mathcal{E}}(u_\delta^j)\|_{L^\infty_{T_*}}^{\frac34}\| u_\delta^j\|_{L^\infty_{T_*}X^1}^{\frac32}+\| u_\delta^j\|_{L^\infty_{T_*}X^1}^3\big) \\
&\lesssim \big[ \delta^{-\frac43}\widetilde{\mathcal{E}}(\phi_\delta)\big]^{\frac34} M_0^{\frac32}+\delta^{-1}M_0^3 \ \to \ 0\qquad (\delta \to \infty).
\end{align*}
Hence, we have \eqref{est:diff-delta-j} by setting
\[ \nu _{M_0}(\delta):=\Big\{ (M_0+M_0^3)\Big( \big[ \delta^{-\frac43}\widetilde{\mathcal{E}}(\phi_\delta)\big]^{\frac34} M_0^{\frac32}+\delta^{-1}M_0^3\Big) +\Big( \big[ \delta^{-\frac43}\widetilde{\mathcal{E}}(\phi_\delta)\big]^{\frac34} M_0^{\frac32}+\delta^{-1}M_0^3\Big)^2\Big\}^{\frac12} .\]

In view of \eqref{est:diff-delta}, the convergence $\| u_\delta- u_\infty \|_{L^\infty_{T_*}X^2}\to 0$ follows once $\| \partial_x^2(u_\delta-u_\infty)\|_{L^\infty_{T_*}L^2}\to 0$ is verified.
To prove the latter convergence, we follow the argument showing continuous dependence in Step~3 of the proof of Theorem~\ref{thm:LWP}.
Using interpolation, \eqref{est:apriori-delta2}, and \eqref{est:diff-delta-j} in place of \eqref{bd:d1}, we have the fixed-$j$ difference estimate corresponding to \eqref{est:cont3}:
\begin{align*}
\| \partial_x^2(u_\delta^j-u_\infty^j)\|_{L^\infty_{T_*}L^2}^2&\leq \big( \| \partial_x^3u_\delta^j\|_{L^\infty_{T_*}L^2}+\| \partial_x^3u_\infty^j\|_{L^\infty_{T_*}L^2}\big) \| u_\delta^j-u_\infty^j\|_{L^\infty_{T_*}X^1}\\
&\lesssim_{M_0} 2^j\big( \| \phi_\delta-\phi_\infty\|_{X^1}+\nu_{M_0}(\delta)\big) .
\end{align*}
Moreover, similarly to Lemma~\ref{lem:fe}~(ii) we can show that
\[ \| \mathbf{c}_{\delta}[\phi_\delta]- \mathbf{c}_{\infty}[\phi_\infty]\|_{\ell^2(\Z_{\geq 0})}\lesssim \| \phi_\delta-\phi_\infty\|_{X^2}+\big[\delta^{-\frac43}\widetilde{\mathcal{E}}(\phi_\delta)\big]^{\frac12}\ \to 0\qquad (\delta\to \infty). \]
Combining them with \eqref{est:apriori-delta4}, we obtain
\begin{align*}
\limsup_{\delta\to \infty}\| \partial_x^2(u_\delta-u_\infty)\|_{L^\infty_{T_*}L^2} &\leq \limsup_{\delta\to\infty}\Big( \| \partial_x^2(u_\delta-u_\delta^j)\|_{L^\infty_{T*}L^2}+\| \partial_x^2(u_\infty-u_\infty^j)\|_{L^\infty_{T*}L^2}+\| \partial_x^2(u_\delta^j-u_\infty^j)\|_{L^\infty_{T_*}L^2}\Big) \\
&\lesssim_{M_0} \limsup_{\delta\to \infty}\Big( \| \mathbf{c}_{\delta}[\phi_\delta]\|_{\ell^2(\Z_{>j})}+\| \mathbf{c}_{\infty}[\phi_\infty]\|_{\ell^2(\Z_{>j})}+\big[ 2^j\big( \| \phi_\delta-\phi_\infty\|_{X^1}+\nu_{M_0}(\delta) \big) \big]^{\frac12}\Big) \\
&\lesssim \| \mathbf{c}_{\infty}[\phi_\infty]\|_{\ell^2(\Z_{>j})}
\ \to 0\qquad (j\to \infty).
\end{align*}
We have thus proven that $T_{\delta,\max}>T_*$ for any $\delta \in [\delta_*,\infty]$ and $\lim\limits_{\delta\to\infty}\| u_\delta-u_\infty\|_{L^\infty_{T_*}X^2}=0$.
In addition, we see from \eqref{est:apriori-delta1'} that $\delta^{-\frac43}\widetilde{\mathcal{E}}(u_\delta(T_*))\to 0$ as $\delta\to \infty$.

Now, we take $T\in (0,T_{\infty,\max})$ arbitrarily.
Since $\| u_\infty (t)\|_{X^2}$ is bounded on $[0,T]$, repeating the above procedure finitely many times we see that there exists $\delta_{**}\in [1,\infty)$ such that $T_{\delta,\max}>T$ for any $\delta\in [\delta_{**},\infty]$ and $\lim\limits_{\delta\to\infty}\| u_\delta-u_\infty \|_{L^\infty_TX^2}=0$. 

Finally, the case $\phi_\infty\in \mathcal{Z}^2$, $\lim\limits_{\delta\to \infty}d^2(\phi_\delta,\phi_\infty)=0$ can be treated by a simpler argument, and we omit the details.
This completes the proof of Theorem~\ref{thm:conv}.
\end{proof}


\section{Global well-posedness}
\label{sec4}

Our goal in this section is to establish Theorem~\ref{thm:GWP}. To do so, we derive the conservation laws for \eqref{gINLS} in the $\mathcal{Z}^2$ framework; in particular, they are meaningful for solutions with nonvanishing boundary conditions at infinity. 
We then show that, for any real parameters $\alpha \neq 0$, $\beta \geq 0$, the Cauchy problem for \eqref{gINLS} is globally well-posed in $\mathcal{Z}^2_*$. 
In addition, when $\alpha =\beta =1$ we obtain a uniform bound for global solutions to \eqref{INLS}.
Our argument also yields the corresponding results for \eqref{gCM}.

\medskip
The main result in this section is the following:
\begin{theorem}\label{thm:conservation}
Let $\delta>0$, $\alpha,\beta \in \mathbb{R}$.
Let $\phi\in \mathcal{Z}^2_\rho$ for some $\rho\geq 0$, and let $u\in C([0,T_{\max});\mathcal{Z}^2_\rho)$ be the unique maximal solution to the Cauchy problem for \eqref{gINLS} with the initial condition $u(0)=\phi$ constructed in Theorem~\ref{thm:LWP}.
\begin{enumerate} 
\item[$\mathrm{(i)}$] The quantity 
\begin{equation*}
\begin{aligned}
\mathcal{H}_1(u):=\int_{\mathbb{R}} \Big\{ |\partial_{x}u|^2&+\alpha (|u|^2-\rho^2)\Im[ \overline{u}\partial_{x}u] +\frac{\beta}{2} (|u|^2-\rho^2)\Td \partial_{x}(|u|^2-\rho^2) \\
&+\alpha^2\Big(\frac{1}{3}|u|^2+\frac{1}{6}\rho^2\Big) (|u|^2-\rho^2)^2\Big\} \,dx 
\end{aligned}
\end{equation*}
is well-defined in $\mathcal{Z}^2_\rho$, and $\mathcal{H}_1(u(t))$ is conserved in time.

\item[$\mathrm{(ii)}$] The quantity
\begin{equation*}
\begin{aligned}
\mathcal{H}_2(u):=\int_{\mathbb{R}} \Big\{ |\partial_x^2u|^2 &+\alpha \Big( 2|\partial_xu|^2-3\partial_x^2(|u|^2)\Big) \Im[\overline{u}\partial_xu]  \\
&+\beta \Big( 2|\partial_xu|^2-\frac34\partial_x^2(|u|^2)\Big) \Td\partial_x(|u|^2-\rho^2) \\
&+\beta \Im[\overline{u}\partial_xu]\Td\partial_x(\Im[\overline{u}\partial_xu]) \Big\} \,dx 
\end{aligned}
\end{equation*}
is well-defined in $\mathcal{Z}^2_\rho$, and $\mathcal{H}_2(u(t))$ satisfies, for $t\in [0,T_{\max})$,
\begin{equation}\label{est:I2}
\mathcal{H}_2(u(t))\lesssim_{\alpha, \beta,\delta} \mathcal{H}_2(u_0)+\int_0^t\Big( \mathcal{E}^1(u(s))+\big[ \mathcal{E}^1(u(s))\big]^{\frac32}\Big) \Big( 1+\| \partial_x^2u(s)\|_{L^2}^2\Big)\,ds.
\end{equation}

\item[$\mathrm{(iii)}$] Assume $\alpha \neq 0$, $\beta \geq 0$ and $\rho>0$. 
Then, the solution $u$ obeys the following a priori bounds:
\begin{align}
\mathcal{E}^1(u(t))&\leq C_1, \label{apriori:e1} \\
\| \partial_x^2u(t)\|_{L^2}^2&\leq C_2e^{C_3t}(1+\|\partial_x^2\phi\|_{L^2}^2) \label{apriori:p2}
\end{align}
for $t\in [0,T_{\max})$, where the constants $C_1,C_2,C_3>0$ depend only on $\alpha,\beta,\delta,\rho$ and $\mathcal{E}^1(\phi)$.
In particular, $T_{\max}=\infty$.
\end{enumerate} 
Furthermore, for \eqref{gCM} we have the same results as above, with $\Td$ replaced by $H$ and with the constants being independent of $\delta$.
\end{theorem}

To derive $\mathcal{H}_1(u)$ and $\mathcal{H}_2(u),$ we proceed by constructing modified energies through a cancellation method \cite{shatah,colliander-kdv}. 
We begin with the time derivative of the leading term (either $\int |\partial_x u|^2 dx$ or $\int |\partial_x^2 u|^2 dx$) and systematically introduce correctors to cancel out the problematic nonlinear terms. 
For $\mathcal{H}_1(u)$, this process is continued to provide a conserved quantity for general $\alpha,\beta$.
Although this is not the case for $\mathcal{H}_2(u)$ in general, in the integrable case we can also derive an exact conserved quantity as in the following corollary.

\begin{corollary}\label{cor:conservation}
Let $\delta>0$ and $\phi\in \mathcal{Z}^2_\rho$ for some $\rho\geq 0$. 
If $\beta =\pm |\alpha|$, the quantity  
\[ \mathcal{H}_2^{\mathrm{INLS}}(u):=\int_{\mathbb{R}} \Big\{ \begin{aligned}[t] &|\partial_x^2u|^2 +\alpha \Big( 2|\partial_xu|^2-3\partial_x^2(|u|^2)\Big) \Im[\overline{u}\partial_xu]  \\
&+\beta \Big( 2|\partial_xu|^2-\frac34\partial_x^2(|u|^2)\Big) \Td\partial_x(|u|^2-\rho^2) +\beta \Im[\overline{u}\partial_xu]\Td\partial_x(\Im[\overline{u}\partial_xu]) \\
&+2\alpha ^2|u|^4|\partial_xu|^2+\alpha ^2|u|^2\Big( \frac32[\partial_x(|u|^2)]^2+\frac12 [\Td\partial_x(|u|^2-\rho^2)]^2\Big) \\
&+\alpha\beta \Im[\overline{u}\partial_xu]\Big( 2|u|^2\Td\partial_x(|u|^2-\rho^2)+ \Td\partial_x(|u|^4-\rho^4)\Big) \\
&+\alpha^2(|u|^6-\rho^6)\Big( \alpha \Im[\overline{u}\partial_xu]+\frac23\beta \Td\partial_x(|u|^2-\rho^2)\Big) +\frac14\alpha^2\beta(|u|^4-\rho^4)\Td\partial_x(|u|^4-\rho^4) \\
&+\alpha^4\Big( \frac15|u|^6+\frac25\rho^2|u|^4+\frac35\rho^4|u|^2+\frac3{10}\rho^6\Big) (|u|^2-\rho^2)^2 \Big\} \,dx \end{aligned} \]
is well-defined in $\mathcal{Z}^2_\rho$ and is a conserved quantity for \eqref{gINLS}.
If $\beta =|\alpha|>0$ and $\rho>0$, the global solution $u\in C(\mathbb{R};\mathcal{Z}^2_{\rho})$ to \eqref{gINLS} satisfies
\begin{equation}\label{GWP-inequality}
\sup_{t\in \mathbb{R}}\mathcal{E}^2(u(t))<\infty.
\end{equation}
Furthermore, for \eqref{gCM} we have the same results as above, with $\Td$ replaced by $H$.
\end{corollary}

\begin{remark}
For vanishing solutions in $H^s$, infinitely many conservation laws for \eqref{gCM} with $\alpha=\pm |\beta|=\pm 1$,  and for \eqref{INLS} have been obtained in \cite{GL} and in \cite{P,CFL}, respectively.
In particular, a conserved quantity at the $H^k(\R)$ level is given by $\| \mathcal{L}_{u(t)}^ku(t)\|_{L^2}^2$, where $\mathcal{L}_u$ is a suitable Lax operator (see \eqref{def:lax} for the case of \eqref{INLS}).
In \cite[Lemma~3.5]{Chen}, Chen found conserved quantities at the $\mathcal{Z}^1$ and $\mathcal{Z}^2$ levels for the defocusing \eqref{gCM} with $\alpha=\beta=1$, which were obtained by expanding $\| \mathcal{L}_uu\|_{L^2}$, $\| \mathcal{L}_u^2u\|_{L^2}$ and ``renormalizing'' them for nonvanishing solutions, i.e., replacing $\Pi(|u|^2)$ with $\Pi(|u|^2-1)+1$, where $\Pi=\frac12(1+iH)$ is the Cauchy-Szeg\H{o} projection and is a bounded operator on $L^2(\R)$.
It is then natural to ask whether conservation laws for nonvanishing solutions to \eqref{INLS} can be obtained by similar renormalization procedure.
However, for \eqref{INLS} the operator $\frac12(1+i\Td)$ corresponding to $\Pi$ becomes unbounded on $L^2(\R)$, and there is no apparent way of renormalizing $\| \mathcal{L}_u^ku\|_{L^2}$ for $u\in \mathcal{Z}^k_\rho$.
Instead, in Theorem~\ref{thm:conservation} and Corollary~\ref{cor:conservation}, we construct conserved quantities directly by using the idea of modified energy method.
\end{remark}

\bigskip
From now on, we only consider \eqref{gINLS}; however, we can apply exactly the same argument to \eqref{gCM} by replacing $\Td$ with $H$.
Before proving Theorem~\ref{thm:conservation}, we collect some useful lemmas.

\begin{lemma}\label{lem:tderi}
Let $u\in C([0,T];\mathcal{Z}^\infty)$ be a smooth solution of \eqref{gINLS}.
Then, 
\begin{align*}
\partial_t(|u|^2)&= \partial_x\Big( 2\Im[\overline{u}\partial_xu] +\alpha |u|^4\Big) ,\\
\partial_t(|\partial_xu|^2)&= \partial_x\Big( 2\Im[(\partial_x\overline{u})\partial_x^2u] +\frac{\alpha}{2}[\partial_x(|u|^2)]^2\Big) +2\alpha |\partial_xu|^2\partial_x(|u|^2) +2\beta \Im[\overline{u}\partial_xu] \Td \partial_x^2(|u|^2),\\
\partial_t(|\partial_x^2u|^2)&= \partial_x\Big( 2\Im[(\partial_x^2\overline{u})\partial_x^3u] +\frac{\alpha}{2}[\partial_x^2(|u|^2)]^2\Big) \\
&\quad +2\alpha |\partial_x^2u|^2\partial_x(|u|^2) +2\alpha \partial_x(|\partial_xu|^2)\cdot \partial_x^2(|u|^2) -2\alpha |\partial_xu|^2\partial_x^3(|u|^2) \\
&\quad +4\beta \Im [(\partial_x\overline{u})\partial_x^2u] \Td \partial_x^2(|u|^2) +2\beta \partial_x(\Im[\overline{u}\partial_xu])\cdot \Td\partial_x^3(|u|^2) ,\\
\partial_t(\Im [\overline{u}\partial_xu])&= 2\partial_x(|\partial_xu|^2) -\frac12 \partial_x^3(|u|^2) +2\alpha \Im[\overline{u}\partial_xu] \partial_x(|u|^2) +\beta |u|^2\Td \partial_x^2(|u|^2) . 
\end{align*}
\end{lemma}

\begin{proof}
The proof follows from a direct calculation.
\end{proof}

\begin{lemma}\label{lem:ImIm}
Let $u\in L^\infty(\R)$ satisfy $\partial_xu\in H^1(\R)$.
Then, we have
\[ \int \Im [(\partial_x\overline{u})\partial_x^2u] \partial_x(\Im [\overline{u}\partial_xu])
=\frac12\int |\partial_x^2u|^2\partial_x(|u|^2) -\frac14\int \partial_x(|\partial_xu|^2)\cdot \partial_x^2(|u|^2) .\]
\end{lemma}
\begin{proof}
Using the identity $\Im(z_1)\Im(z_2)=\Re(z_1\overline{z_2})-\Re(z_1)\Re(z_2)$ ($z_1,z_2\in \mathbb{C}$), we have
\begin{align*}
&\int \Im [(\partial_x\overline{u})\partial_x^2u] \partial_x(\Im [\overline{u}\partial_xu]) =\int |\partial_x^2u|^2\Re [u\partial_x\overline{u}] -\int \Re[(\partial_x\overline{u})\partial_x^2u] \Re [\overline{u}\partial_x^2u] \\
&=\int |\partial_x^2u|^2\cdot \frac12\partial_x(|u|^2) -\int \frac12 \partial_x(|\partial_xu|^2)\cdot \Big( \frac12 \partial_x^2(|u|^2)-|\partial_xu|^2\Big) \\
&=\frac12\int |\partial_x^2u|^2\partial_x(|u|^2) -\frac14\int \partial_x(|\partial_xu|^2)\cdot \partial_x^2(|u|^2) ,
\end{align*}
as desired.
\end{proof}

\medskip
We are now in a position to prove Theorem \ref{thm:conservation}. 

\begin{proof}[Proof of Theorem~\ref{thm:conservation}]
First, it is easy to see that $\mathcal{H}_1(u)$ and $\mathcal{H}_2(u)$ are well-defined for $u\in \mathcal{Z}^2_\rho$, noticing $u\in L^\infty(\R)$, $\partial_xu\in H^1(\R)$, $|u|^2-\rho^2\in H^1(\R)$, and that $\Td\partial_x$ maps $H^1(\R)$ into $L^2(\R)$ by \eqref{est:Tdelta}.
Moreover, we may assume that $u\in \mathcal{Z}^\infty_\rho$ by a standard approximation argument, which relies on continuity of the flow map and persistence of regularity in the local well-posedness Theorem~\ref{thm:LWP}.

For (i), we begin by observing, from Lemma~\ref{lem:tderi}, that
\[ \partial_t \int |\partial_xu|^2 =2\alpha \int |\partial_xu|^2\partial_x(|u|^2) +2\beta \int \Im[\overline{u}\partial_xu] \Td \partial_x^2(|u|^2) .\]
To cancel out two terms on the right-hand side, we introduce two correctors 
\[ I_1^{(1)}(u):=\int (|u|^2-\rho^2)\Im[ \overline{u}\partial_{x}u],\qquad I_1^{(2)}(u):= \int (|u|^2-\rho^2)\Td \partial_{x}(|u|^2-\rho^2) .\]
Using Lemma~\ref{lem:tderi} and integration by parts, we see that
\begin{align*}
\partial_tI_1^{(1)}(u)&=-2\int |\partial_xu|^2\partial_x(|u|^2) -2\alpha \int |u|^2(|u|^2-\rho^2) \partial_x(\Im[\overline{u}\partial_xu]) \\
&\quad -\beta \int \partial_x(|u|^4)\cdot \Td \partial_x(|u|^2-\rho^2) ,\\
\partial_tI_1^{(2)}(u)&=-4\int \Im[\overline{u}\partial_xu] \Td\partial_x^2(|u|^2) +2\alpha \int\partial_x(|u|^4)\cdot \Td \partial_x(|u|^2-\rho^2) .
\end{align*}
Here, we have used the symmetry of $\Td\partial_x$ \eqref{id:Tdeltasymm} to verify
\[ \int \partial_x(|u|^2)\cdot \Td\partial_x(|u|^2-\rho^2)=\int \partial_x(|u|^2-\rho^2)\cdot \Td\partial_x(|u|^2-\rho^2)=0.\]
From these calculations, we see that
\[ \partial_t\Big( \int |\partial_xu|^2 +\alpha I_1^{(1)}(u)+\frac{\beta}{2}I_1^{(2)}(u)\Big) =-2\alpha ^2\int |u|^2(|u|^2-\rho^2)\partial_x(\Im[\overline{u}\partial_xu]).\]
To cancel out the right-hand side, we further add two correctors
\[ I_1^{(3)}(u):= \int (|u|^2-\rho^2)^3,\qquad I_1^{(4)}(u):=\int (|u|^2-\rho^2)^2 .\]
Using Lemma~\ref{lem:tderi} and integration by parts, we see that
\begin{align*}
\partial_tI_1^{(3)}(u)&=6\int (|u|^2-\rho^2)^2\partial_x (\Im[\overline{u}\partial_xu]),\\
\partial_tI_1^{(4)}(u)&=4\int (|u|^2-\rho^2)\partial_x (\Im[\overline{u}\partial_xu]) .
\end{align*}
Here, we have calculated as follows: for $k\geq 1$, 
\begin{align*}
\int (|u|^2-\rho^2)^k\partial_x(|u|^4) 
&=\int (|u|^2-\rho^2)^k\partial_x\Big( (|u|^2-\rho^2)^2 +2\rho^2(|u|^2-\rho^2)+\rho^4 \Big) \\
&=\int \partial_x\Big( \frac{2}{k+2}(|u|^2-\rho^2)^{k+2}+\frac{2\rho^2}{k+1}(|u|^2-\rho^2)^{k+1}\Big) =0.
\end{align*}
From the above calculations, we have
\[ 2\alpha ^2\partial_t\Big( \frac16I_1^{(3)}(u)+\frac{\rho^2}4I_1^{(4)}(u)\Big) =2\alpha ^2\int |u|^2(|u|^2-\rho^2)\partial_x(\Im[\overline{u}\partial_xu]).\]
Consequently, the quantity
\[ \mathcal{H}_1(u):=\int |\partial_xu|^2 +\alpha I_1^{(1)}(u)+\frac{\beta}{2}I_1^{(2)}(u)+\frac{\alpha^2}{3}I_1^{(3)}(u)+\frac{\alpha^2\rho^2}{2}I_1^{(4)}(u) \]
is conserved in time.
This finishes the proof of (i).

In a similar manner, to compute $\mathcal{H}_2$, we begin with
\begin{align*}
\partial_t\int |\partial_x^2u|^2 &=2\alpha \int |\partial_x^2u|^2\partial_x(|u|^2) +4\alpha \int \partial_x(|\partial_xu|^2)\cdot \partial_x^2(|u|^2) \\
&\quad +4\beta \int \Im [(\partial_x\overline{u})\partial_x^2u] \Td \partial_x^2(|u|^2) +2\beta \int \partial_x(\Im[\overline{u}\partial_xu])\cdot \Td\partial_x^3(|u|^2) ,
\end{align*}
and to cancel out the terms on the right-hand side we introduce five correctors%
\footnote{~We can predict these terms from the terms of degree four in $E_2(u)$, which can be obtained by expanding the integral (1.4) with $n=4$.}
\begin{gather*}
I_2^{(1)}(u):=\int |\partial_xu|^2\Im[\overline{u}\partial_xu] ,\qquad
I_2^{(2)}(u):=\int \partial_x^2(|u|^2)\cdot \Im[\overline{u}\partial_xu] ,\\
I_2^{(3)}(u):=\int |\partial_xu|^2\Td\partial_x(|u|^2-\rho^2) ,\qquad
I_2^{(4)}(u):=\int \partial_x^2(|u|^2)\cdot \Td\partial_x(|u|^2-\rho^2) ,\\
I_2^{(5)}(u):=\int \Im[\overline{u}\partial_xu]\Td\partial_x(\Im[\overline{u}\partial_xu]) .
\end{gather*}
By a calculation involving Lemmas~\ref{lem:tderi}--\ref{lem:ImIm}, we see that
\begin{align*}
\partial_tI_2^{(1)}(u)&= -\int |\partial_x^2u|^2\partial_x(|u|^2) +\int \partial_x(|\partial_xu|^2)\cdot \partial_x^2(|u|^2) \ +R_1(u),\\
\partial_tI_2^{(2)}(u)&=2\int \partial_x(|\partial_xu|^2)\cdot \partial_x^2(|u|^2) \ +R_2(u) ,\\
\partial_tI_2^{(3)}(u)&=-2\int \Im[\partial_x\overline{u}\cdot \partial_x^2u]\Td\partial_x^2(|u|^2) -2\int \Im[\overline{u}\partial_xu]\Td\partial_x^2(|\partial_xu|^2) \ +R_3(u),\\
\partial_tI_2^{(4)}(u)&=4\int \partial_x(\Im[\overline{u}\partial_xu])\cdot \Td\partial_x^3(|u|^2) \ +R_4(u),\\
\partial_tI_2^{(5)}(u)&=\int \partial_x(\Im[\overline{u}\partial_xu])\cdot \Td\partial_x^3(|u|^2) +4\int \Im[\overline{u}\partial_xu]\Td\partial_x^2(|\partial_xu|^2) \ +R_5(u).
\end{align*}
Here, each of the remainder terms $R_j(u)$ ($j=1,\dots,5$) is a linear combination of integrals with six $u$'s and four spatial differentials, and also of the form 
\begin{gather*}
\text{either}\qquad \int fh_1h_2\qquad \text{or}\qquad \int g_1g_2h
\intertext{with}
f\in \big\{ 1,|u|^2\big\},\qquad g_1,g_2\in \big\{ \partial_x(|u|^2),\Im [\overline{u}\partial_xu],\Td\partial_x(|u|^2) \big\},\\
h,h_1,h_2\in \big\{ |\partial_xu|^2,\partial_x^2(|u|^2), \partial_x^2(|u|^4), \partial_x(\Im [\overline{u}\partial_xu]) ,\Td\partial_x^2(|u|^2), \Td\partial_x^2(|u|^4), \Td\partial_x(\Im [\overline{u}\partial_xu]) \big\} .
\end{gather*}
By using H\"older, Gagliardo-Nirenberg $\| \partial_xu\|_{L^\infty}\leq \| \partial_xu\|_{L^2}^{1/2}\| \partial_x^2u\|_{L^2}^{1/2}$, and the estimate  \eqref{est:Tdelta}, we estimate these remainder terms as
\begin{equation}\label{remainder-conservation laws}
\begin{aligned}
|R_j(u)|&\lesssim_{\alpha,\beta ,\delta} \Big( \| u\|_{L^\infty}+\| \partial_xu\|_{L^2}\Big) ^4 \Big( \| u\|_{L^\infty}^2+\| \partial_xu\|_{L^2}^2+\| \partial_x^2u\|_{L^2}^2 \Big) \\
&\lesssim \Big(\mathcal{E}^1(u)+\big[\mathcal{E}^1(u)\big]^{\frac32}\Big) \Big( 1+\| \partial_x^2u\|_{L^2}^2\Big) \notag.
\end{aligned}
\end{equation}
We also see that 
\[ \mathcal{H}_2(u):= \int |\partial_x^2u|^2 +2\alpha I_2^{(1)}(u)-3\alpha I_2^{(2)}(u)+2\beta I_2^{(3)}(u) -\frac34\beta I_2^{(4)}(u) +\beta I_2^{(5)}(u)\]
is the only choice of the coefficients to cancel out the right-hand side of $\partial_t\int |\partial_x^2u|^2$ and yield
\[\partial_t |\mathcal{H}_2(u)| \lesssim _{\alpha,\beta}\sum_{j=1}^5|R_j(u)| \lesssim _{\alpha,\beta,\delta} \Big( \mathcal{E}^1(u)+\big[\mathcal{E}^1(u)\big]^{\frac32}\Big) \Big(1+\| \partial_x^2u\|_{L^2}^2\Big) .\]
Integrating in $t$, we obtain \eqref{est:I2} and thus (ii).

It remains to prove (iii).
By Young's inequality, we see that
\[ \Big| \alpha \int (|u|^2-\rho^2)\Im[\overline{u}\partial_xu]\Big| \leq \frac56\int |\partial_xu|^2+\frac{3}{10}\alpha^2\int |u|^2(|u|^2-\rho^2)^2 .\]
Thus, noticing the positivity of $\Td\partial_x$ in \eqref{id:Tdeltasymm} and the assumption $\beta \geq 0$, we see that
\[ \mathcal{H}_1(u)\geq \frac16\| \partial_xu\|_{L^2}^2+\frac{\alpha^2\rho^2}{6}\| |u|^2-\rho^2\|_{L^2}^2.\]
Using the inequality \eqref{est:Linfty} and that $\alpha^2\rho^2>0$, we conclude
\begin{equation}\label{est:I1below}
\| \partial_xu\|_{L^2}^2\lesssim \mathcal{H}_1(u),\quad \| |u|^2-\rho^2\|_{L^2}\lesssim_{\alpha,\rho}\mathcal{H}_1(u)^{\frac12},\quad \| u\|_{L^\infty}^2\lesssim_{\alpha,\rho}\mathcal{H}_1(u)^{\frac23}+1.
\end{equation}
On the other hand, using \eqref{est:Tdelta}--\eqref{est:TLonZ} we easily estimate $\mathcal{H}_1(\phi)$ as
\begin{equation}\label{est:I1above}
\begin{aligned}
\mathcal{H}_1(\phi)&\lesssim_{\alpha,\beta,\rho,\delta} \| \partial_x\phi\|_{L^2}^2+\| |\phi|^2-\rho^2\|_{L^2}\| \phi\|_{L^\infty}\| \partial_x\phi\|_{L^2}+\Big( \| \phi\|_{L^\infty}^2+1\Big) \| |\phi|^2-\rho^2\|_{L^2}^2 \\
&\lesssim \big[\mathcal{E}^1(\phi)\big]^{\frac12}+\big[\mathcal{E}^1(\phi)\big]^{\frac32}.
\end{aligned}
\end{equation}
From \eqref{est:I1below}, \eqref{est:I1above} and the conservation of $\mathcal{H}_1(u(t))$ in (i), we obtain \eqref{apriori:e1}.

Finally, we prove \eqref{apriori:p2}.
Using \eqref{est:Tdelta}--\eqref{est:TLonZ}, H\"older's inequality and the a priori bound \eqref{apriori:e1}, we see that
\begin{gather*}
\| \partial_x^2(|u|^2)\|_{L^2}+\| |\partial_xu|^2\|_{L^2}+\| \Td\partial_x(\Im[\overline{u}\partial_xu])\|_{L^2}\lesssim 1+\| \partial_x^2u\|_{L^2},\\
\| \Im[\overline{u}\partial_xu]\|_{L^2}+\| \Td\partial_x(|u|^2)\|_{L^2}\lesssim 1,
\intertext{and then, by Young's inequality, that}
\Big| \mathcal{H}_2(u)-\int |\partial_x^2u|^2\Big| \leq \frac12\int |\partial_x^2u|^2+C_0,
\end{gather*}
where $C_0>0$ and other implicit constants may depend on $\alpha,\beta,\delta,\rho, \mathcal{E}^1(\phi)$. 
In particular, we have
\begin{equation}\label{est:equivI2}
\int |\partial_x^2u|^2\leq 2\mathcal{H}_2(u)+2C_0,\qquad \mathcal{H}_2(u)\leq \frac32\int|\partial_x^2u|^2+C_0.
\end{equation}
The desired a priori bound \eqref{apriori:p2} follows from \eqref{est:I2}, \eqref{apriori:e1}, \eqref{est:equivI2} and Gronwall's inequality.

This is the end of the proof of Theorem~\ref{thm:conservation}.
\end{proof}

\medskip
At the end of this section, we give an outline of the proof of Corollary~\ref{cor:conservation}.

\begin{proof}[Outline of proof of Corollary~\ref{cor:conservation}]
 In the integrable case of \eqref{gINLS} where $\beta =\pm |\alpha|$, we can further repeat the energy correction process (as in the proof of Theorem~\ref{thm:conservation} (ii)) to obtain an exactly conserved quantity $\mathcal{H}^{\mathrm{INLS}}_2(u)$, by introducing the following new correctors of order six with two spatial differentials
\begin{gather*}
I_2^{(6)}(u):=\int |u|^4|u_x|^2 ,\qquad
I_2^{(7)}(u):=\int |u|^2\big[(|u|^2)_x\big]^2 ,\\
I_2^{(8)}(u):=\int |u|^2\Im[\overline{u}u_x]\Td\partial_x(|u|^2-\rho^2) ,\qquad
I_2^{(9)}(u):=\int \Im[\overline{u}u_x]\Td\partial_x(|u|^4-\rho^4) ,\\
I_2^{(10)}(u):=\int |u|^2\big[ \Td \partial_x(|u|^2-\rho^2)\big]^2,
\end{gather*}
the correctors of order eight with one differential
\begin{gather*}
I_2^{(11)}(u):=\int (|u|^6-\rho^6)\Im [\overline{u}u_x],\\
I_2^{(12)}(u):=\int (|u|^6-\rho^6)\Td\partial_x(|u|^2-\rho^2),\qquad
I_2^{(13)}(u):=\int (|u|^4-\rho^4)\Td\partial_x(|u|^4-\rho^4),\end{gather*}
and finally, the corrector of order ten with no differential
\begin{align*}
I_2^{(14)}(u):=D_1\int (|u|^2-\rho^2)^5 +D_2\rho^2\int (|u|^2-\rho^2)^4+D_3\rho^4\int (|u|^2-\rho^2)^3+D_4\rho^6\int (|u|^2-\rho^2)^2\\
(D_1,\dots ,D_4\in \R), 
\end{align*}
and successively determining the coefficients to cancel the remainder in \eqref{remainder-conservation laws}.

Second, to prove \eqref{GWP-inequality}, we recall that in the proof of \eqref{apriori:p2} in Theorem~\ref{thm:conservation}, we have already seen that 
\[ \sum_{j=1}^5 |I_2^{(j)}(u)|\lesssim 1+\| \partial_x^2u\|_{L^2}. \]
Similarly, using H\"older's inequality and \eqref{est:Tdelta}, we see that
\[ \sum_{j=6}^{14}|I_2^{(j)}(u)|\lesssim 1.\]
Here, the implicit constants depend on $\alpha,\delta,\rho$ and $\mathcal{E}^1(u)$.
As we did for \eqref{est:equivI2}, by Young's inequality we deduce
\[ \Big| \mathcal{H}_2^{\mathrm{INLS}}(u(t))-\int |\partial_x^2u(t)|^2\Big| \leq \frac12\int |\partial_x^2u(t)|^2 +C_4\big(\alpha,\delta,\rho,\mathcal{E}^1(u(t))\big)\qquad (t\in [0,T]) \]
for some constant $C_4>0$.
Then, from the a priori bound \eqref{apriori:e1} for $\mathcal{E}^1(u(t))$ proven in Theorem~\ref{thm:conservation} (here we use the defocusing assumption $\beta =|\alpha|>0$) and conservation of $\mathcal{H}_2^{\mathrm{INLS}}(u(t))$, we infer that
\begin{align*}
\| \partial_x^2u(t)\|_{L^2}^2&\leq 2\mathcal{H}_2^{\mathrm{INLS}}(u(t))+2C_4\big(\alpha,\delta,\rho,\mathcal{E}^1(u(t))\big) \\
&\leq 2\mathcal{H}_2^{\mathrm{INLS}}(\phi)+2C_4\big(\alpha,\delta,\rho,C_1(\alpha,\delta,\rho,\mathcal{E}^1(\phi)\big) \\
&\leq 3\|\partial_x^2\phi\|_{L^2}^2+2C_4\big(\alpha,\delta,\rho,\mathcal{E}^1(\phi)\big)+2C_4\big(\alpha,\delta,\rho,C_1(\alpha,\delta,\rho,\mathcal{E}^1(\phi))\big),
\end{align*}
where $C_1>0$ is the constant in \eqref{apriori:e1}.
This completes the proof.
\end{proof}

\appendix 

\section{Fourier transform of the hyperbolic cotangent function}
\label{app:Fourier}

We give a proof of \eqref{T[delta]-Fourier}. 
To this end, we will use the following lemma. 
\begin{lemma}\label{25/7/19/11:17}
We have
\begin{equation}\label{25/7/19/11:18}
\mathcal{F}\bigg[ \frac{1}{e^{x}+1} -\frac{1}{e^{-x}+1}\bigg](\xi) = 2i \pi \,\pv\operatorname{cosech}{(\pi \xi)} ,
\end{equation}
where $\pv$ denotes the principal value.
\end{lemma}
\begin{proof}[Proof of Lemma \ref{25/7/19/11:17}]
Let $0<\varepsilon <1/2$ be a constant to be taken $\varepsilon \to 0$.  
We consider the perimeter of the rectangle $\{ z \in \mathbb{C} \colon -R \leq \Re[z] \leq R,~0\leq \Im[z]\leq 2\pi \}$. Then, an elementary complex analysis involving the residue theorem shows that 
\[ \begin{aligned} 
\mathcal{F}\left[\frac{e^{\varepsilon x}}{e^{x}+1}\right](\xi)
&= \int_{-\infty}^{\infty} \frac{e^{(\varepsilon -i\xi) x}}{e^{x}+1} \,dx \\
&=2 i \pi\,\mathrm{Res} \left(\frac{e^{(\varepsilon -i\xi)z}}{e^{z}+1}; i\pi \right)+\lim_{R\to \infty}\int_{-R}^{R} \frac{e^{(\varepsilon -i\xi) (x+2 i \pi )}}{e^{x+2i \pi}+1} \,dx \\
&\quad -\lim_{R\to \infty}\int_{0}^{2\pi} \frac{e^{(\varepsilon -i\xi) (R+iy)}}{e^{R+iy}+1} i \,dy+\lim_{R\to \infty}\int_{0}^{2\pi} \frac{e^{(\varepsilon -i\xi) (-R+iy)}}{e^{-R+iy}+1} i\,dy\\
&=-2i \pi e^{\pi(\xi + i\varepsilon)}+e^{2\pi(\xi + i\varepsilon)}\int_{-\infty}^{\infty} \frac{e^{(\varepsilon -i\xi) x}}{e^{x}+1} \,dx \\
&=-2i \pi e^{\pi(\xi + i\varepsilon)}+ e^{2\pi(\xi + i\varepsilon)}\mathcal{F}\left[ \frac{e^{\varepsilon x}}{e^{x}+1}\right](\xi),
\end{aligned} \]
which implies that  
\begin{equation}\label{25/7/19/10:53}
\mathcal{F}\left[\frac{e^{\varepsilon x}}{e^{x}+1}\right](\xi)=\frac{-2i \pi e^{\pi(\xi + i\varepsilon)}}{1-e^{2\pi(\xi + i\varepsilon)}}=\frac{2i \pi }{e^{\pi(\xi + i\varepsilon)}-e^{-\pi(\xi + i\varepsilon)}}.
\end{equation}
Moreover, by $\mathcal{F}[f(-x)](\xi) = \mathcal{F}[f](-\xi)$, we see that 
\begin{equation}\label{25/7/19/10:7}
\mathcal{F}\left[\frac{e^{-\varepsilon x}}{e^{-x}+1}\right](\xi)=
\mathcal{F}\left[\frac{e^{\varepsilon x}}{e^{x}+1}\right](-\xi)=
\frac{2i \pi }{e^{-\pi(\xi - i\varepsilon)}-e^{\pi(\xi - i\varepsilon)}}.
\end{equation}
By \eqref{25/7/19/10:53}, \eqref{25/7/19/10:7}, and  
\[ \lim_{\varepsilon \downarrow 0}\left( \frac{e^{\varepsilon x}}{e^{x}+1}-\frac{e^{-\varepsilon x}}{e^{-x}+1}\right) =\frac{1}{e^{x}+1} -\frac{1}{e^{-x}+1}\qquad \mbox{for all $x\in \mathbb{R}$}, \]
it suffices for \eqref{25/7/19/11:18} to show that 
\begin{equation}\label{25/7/19/11:25}
\begin{aligned}
F_{\varepsilon}(\xi)
&:= \frac{1}{2i \pi}\mathcal{F}\left[\frac{e^{\varepsilon x}}{e^{x}+1}- \frac{e^{-\varepsilon x}}{e^{-x}+1}\right](\xi) \\
&\;=\frac{1}{e^{i \pi \varepsilon} e^{\pi \xi }-e^{-i \pi \varepsilon}e^{-\pi \xi}}-\frac{1}{e^{i\pi \varepsilon}e^{-\pi \xi}-e^{-i\pi\varepsilon }e^{\pi \xi}} \\
&\;\to \pv \operatorname{cosech}{(\pi \xi)}\quad \mbox{in the tempered distribution sense, as $\varepsilon \downarrow 0$}. 
\end{aligned}
\end{equation}
Let $\phi$ be a Schwartz function on $\mathbb{R}$, and $0< r \leq 1$ a constant to be taken $r\to 0$. 
Since $F_{\varepsilon}$ is an odd function, we see that 
\begin{equation}\label{25/7/19/11:39}
\int_{\mathbb{R}} F_{\varepsilon}(\xi) \phi(\xi) \,d\xi 
=\int_{|\xi|>r} F_{\varepsilon}(\xi) \phi(\xi) \,d\xi +\int_{-r}^{r} F_{\varepsilon}(\xi) \left\{ \phi(\xi)-\phi(0)\right\} \,d\xi .
\end{equation}
Note that $F_{\varepsilon}(\xi)$ converges to $\operatorname{cosech}{(\pi \xi)}$ for all $x\in \mathbb{R}\setminus \{ 0\}$ as $\varepsilon \downarrow 0$. 
Moreover, we see that 
\begin{equation}\label{25/7/19/11:48}
\begin{aligned}
\left| F_{\varepsilon}(\xi) \right|
&=\left| \frac{1}{e^{i \pi \varepsilon} e^{\pi \xi }-e^{-i \pi \varepsilon}e^{-\pi \xi}}-\frac{1}{e^{i\pi \varepsilon}e^{-\pi \xi}-e^{- i\pi \varepsilon}e^{\pi \xi}}\right| \\
&=\left| \frac{(e^{i\pi \varepsilon} + e^{-i \pi \varepsilon}) (e^{-\pi \xi} - e^{\pi \xi})}{e^{2i \pi \varepsilon} +e^{-2i \pi \varepsilon} -( e^{2\pi \xi } + e^{-2\pi \xi})}\right| =\left| \frac{2 \cos{(\pi \varepsilon)}(e^{-\pi \xi} - e^{\pi \xi})}{2 \cos{(2 \pi \varepsilon)} - (e^{2\pi \xi } + e^{-2\pi \xi})}\right| \\
&\leq \frac{2 |e^{-\pi \xi} - e^{\pi \xi}|}{e^{2\pi \xi } + e^{-2\pi \xi} - 2}=\frac{2 |e^{-\pi \xi} - e^{\pi \xi}|}{|e^{\pi \xi } - e^{-\pi \xi}|^{2} 
}=\frac{2 }{|e^{\pi \xi } - e^{-\pi \xi}|}.  
\end{aligned} 
\end{equation}
By \eqref{25/7/19/11:48}, we see that $F_{\varepsilon}$ is uniformly bounded with respect to $\varepsilon$ for $|\xi|>r$; for instance, we have 
\[ \left| F_{\varepsilon}(\xi) \right| \leq \frac{2}{e^{\pi r } - e^{-\pi r}}. \]
Thus, Lebesgue's dominated convergence theorem shows that 
\begin{equation}\label{25/7/19/12:37}
\lim_{\varepsilon \downarrow 0}\int_{|\xi|>r} F_{\varepsilon}(\xi) \phi(\xi) \,d\xi 
=\int_{|\xi|>r}\operatorname{cosech}{(\pi \xi)}\phi(\xi)\,d\xi. 
\end{equation}
Furthermore, by \eqref{25/7/19/11:48}, we see that
\begin{equation}\label{25/7/20/10:29}
\lim_{\varepsilon \downarrow 0}\left| \int_{-r}^{r} F_{\varepsilon}(\xi) \left\{ \phi(\xi) -\phi(0) \right\} d\xi \right| 
\leq \int_{-r}^{r} \frac{2e^{\pi \xi} |\xi|}{|e^{2\pi \xi} - 1|}\left| \frac{\phi(\xi) -\phi(0)}{\xi} \right| d\xi 
\lesssim |r| \| \partial_\xi \phi \|_{L^{\infty}},
\end{equation}
where the implicit constant is independent of $\varepsilon$ and $r$. 
Putting \eqref{25/7/19/11:39}, \eqref{25/7/19/12:37} and \eqref{25/7/20/10:29} together, we see that 
\begin{equation}\label{25/7/19/12:55}
\begin{aligned}
\lim_{\varepsilon \downarrow 0}\int_{\mathbb{R}} F_{\varepsilon}(\xi) \phi(\xi) \,d\xi 
&=\lim_{r\downarrow 0} \lim_{\varepsilon \downarrow 0}\bigg\{ 
\int_{|\xi|>r} F_{\varepsilon}(\xi) \phi(\xi) \,d\xi +
\int_{-r}^{r} F_{\varepsilon}(\xi) \left\{ \phi(\xi) -\phi(0) \right\} d\xi \bigg\} \\
&=\lim_{r\downarrow 0}\int_{|\xi|>r}\operatorname{cosech}{(\pi \xi)}\phi(\xi)\,d\xi=\langle \pv\operatorname{cosech}{(\pi \xi)}, \phi \rangle.
\end{aligned} 
\end{equation}
Thus, we have proven \eqref{25/7/19/11:25} and the lemma. 
\end{proof}

Now, we are in a position to prove \eqref{T[delta]-Fourier}. 
\begin{proof}[Proof of \eqref{T[delta]-Fourier}]
It is sufficient to prove that 
\begin{equation}\label{25/7/12/7:1}
\mathcal{F}[\pv\coth{(\cdot)}](\xi) = -i \pi \, \pv\coth{\Big(\frac{\pi \xi}{2}\Big)} .
\end{equation}
Indeed, \eqref{25/7/12/7:1} together with basic properties of Fourier transformations show that for any Schwartz function $f$ on $\mathbb{R}$, 
\begin{align*}
\mathcal{F}[\Td f](\xi)
&= \frac{1}{2\delta} \mathcal{F}\left[ \pv \coth{\left( \frac{\pi \cdot}{2\delta}\right)}* f\right](\xi)=\frac{1}{2\delta} \mathcal{F}\left[\pv\coth{\left( \frac{\pi \cdot}{2\delta}\right)}\right](\xi)
\mathcal{F}[f](\xi) \\
&=\frac{1}{\pi} \mathcal{F}\left[\pv \coth{(\cdot)}\right]\Big(\frac{2\delta \xi}{\pi}\Big)\mathcal{F}[f](\xi)=-i\, \pv \coth{\left(\delta \xi \right)}\mathcal{F}[f](\xi).
\end{align*}

Note that 
\[ \frac{1}{e^{x}+1} - \frac{1}{e^{-x}+1}=\frac{e^{-\frac{x}{2}}-e^{\frac{x}{2}}}{e^{\frac{x}{2}}+e^{-\frac{x}{2}}}=-\tanh{\left( \frac{x}{2}\right)}.\]
Hence, by Lemma \ref{25/7/19/11:17}, we see that 
\begin{equation}\label{25/7/20/10:52}
\mathcal{F}\left[ \tanh{\left(\frac{x}{2}\right)} \right]\left(\xi \right)
=-\mathcal{F}\left[ \frac{1}{e^{x}+1} - \frac{1}{e^{-x}+1}\right]\left( \xi\right)=-2i \pi\, \pv\operatorname{cosech}{\left( \pi \xi \right)}.
\end{equation}
Furthermore, by \eqref{25/7/20/10:52}, and $\mathcal{F}^{2}[f](x) = 2\pi f(-x)$, we see that 
\begin{equation}\label{25/7/20/12:1}
\begin{aligned}
\mathcal{F}\left[ \pv\operatorname{cosech}{\left( x \right)}\right](\xi) 
&=\frac{i}{2}\mathcal{F}\left[ -2 i \pi\, \pv\operatorname{cosech}{\left(\pi x \right)}\right]\left(\pi \xi \right)\\
&= \frac{i}{2}\mathcal{F}^{2}\left[ \tanh{\left(\frac{x}{2}\right)} \right]\left(\pi \xi\right)=i\pi \tanh{\Big( \frac{-\pi \xi}{2}\Big)}=-i\pi \tanh{\Big( \frac{\pi \xi}{2}\Big)}.
\end{aligned} 
\end{equation}
Note that 
\begin{equation}\label{25/7/20/12:19}
\begin{aligned}
\coth{(x)} 
&=\frac{e^{x}+e^{-x}}{e^{x}-e^{-x}}=\frac{(e^{\frac{x}{2}} - e^{-\frac{x}{2}})^{2} + 2}{e^{x}-e^{-x}}=\frac{e^{\frac{x}{2}} - e^{-\frac{x}{2}}}{e^{\frac{x}{2}}+ e^{-\frac{x}{2}}}+\frac{2}{ e^{x}- e^{-x}}\\
&=\tanh{\left(\frac{x}{2} \right)} + \operatorname{cosech}{(x)}.
\end{aligned}
\end{equation}
Moreover, we see that 
\begin{equation}\label{25/7/20/13:3}
\operatorname{cosech}{\left(x \right)} + \coth{(x)}
=\frac{2+ e^{x} + e^{-x}}{e^{x}- e^{-x}}= \frac{(e^{\frac{x}{2}} + e^{-\frac{x}{2}})^{2}}{e^{x}-e^{-x}}=\coth{\left( \frac{x}{2}\right)}. 
\end{equation}
Using \eqref{25/7/20/12:19} twice, 
 \eqref{25/7/20/10:52}, \eqref{25/7/20/12:1}, and 
 \eqref{25/7/20/13:3}, we see that
\begin{equation}\label{25/7/20/12:55}
\begin{aligned} 
\mathcal{F}\left[\pv \coth{(x)}\right](\xi) 
&=\mathcal{F}\left[ \tanh{\left(\frac{x}{2} \right)}\right](\xi) + \mathcal{F}\left[\pv\operatorname{cosech}{(x)}\right](\xi)\\
&=-2i \pi\, \pv\operatorname{cosech}{\left( \pi \xi \right)}-i\pi \tanh{\Big( \frac{\pi \xi}{2}\Big)}\\
&=-i \pi\, \pv\operatorname{cosech}{\left( \pi \xi \right)}-i\pi\,\pv \coth{(\pi \xi)}\\
&=-i\pi\,\pv \coth{\Big( \frac{\pi \xi}{2} \Big)}.
\end{aligned}
\end{equation}
Thus, we have proven \eqref{25/7/12/7:1}; thus \eqref{T[delta]-Fourier} is true. 
\end{proof}


 \section{Proof of Lemma~\ref{lem:fe} and \eqref{est:fe}}
\label{sec:app-fe}

\begin{proof}[Proof of Lemma~\ref{lem:fe}]
(i) The first inequality follows immediately from the definition of $c_{\delta,j}[f]$. 
We also see that $c_{\delta,j\pm 1}[f]\leq 2^{\eps}c_{\delta,j}[f]$, which implies the second inequality.
For the last one, we easily see $\mathcal{E}^2_\delta(f)\lesssim \| \mathbf{c}_\delta[f]\|_{\ell^2}^2$, while the converse is shown by the (discrete) Young inequality.

(ii) Using the inequality $\big| \| \phi \|_{L^2}^{1/2} -\| \psi \|_{L^2}^{1/2}\big|\leq \| \phi -\psi \|_{L^2}^{1/2}$, we have 
\[ |c_{\delta,j}[f] - c_{\delta,j}[g]| \leq \sum_{k=0}^{\infty} 2^{-\eps |j-k|}
\big\| P_{k}\partial_x(f-g)\big\|_{H^1}+2^{-\eps j}\Big( \| f-g\|_{L^\infty} +\delta^{-\frac23}\big\| (|f|^2-\rho_f^2)-(|g|^2-\rho_g^2)\big\|_{L^2}^{\frac12}\Big) .\]
The claim follows by Young's inequality as before. 

(iii) By Bernstein's inequality, (i), and that $\eps \in (0,1)$, we see that 
\[ \| P_{\leq j} \partial_x f \|_{H^2} 
\lesssim \sum_{k=0}^{j} 2^{k} 
\| P_{k} \partial_{x} f \|_{H^1} \leq \sum_{k=0}^{j} 2^{k} 2^{\eps (j-k)} 
c_{\delta,j}[f] \lesssim 2^{j} c_{\delta,j}[f]. \]
On the other hand, we see that $\| P_{\leq j}f\|_{L^\infty}\leq \| f\|_{L^\infty}\leq 2^{\eps j}c_{\delta,j}[f]\leq 2^jc_{\delta,j}[f]$, and the claim follows from these estimates. 
Thus, we have completed the proof of Lemma~\ref{lem:fe}.
\end{proof}

\begin{proof}[Proof of \eqref{est:fe}]
By Bernstein's inequality, the $X^3$ bound from \eqref{bd:Ek}, and Lemma~\ref{lem:fe} (iii), we see that 
\begin{align*}
\| P_k\partial_x^2(u^j-u^{j-1})\|_{L^\infty_{T_*}L^2}&\lesssim 2^{-k}\big( \| \partial_xu^j\|_{L^\infty_{T_*}H^2}+ \| \partial_xu^{j-1}\|_{L^\infty_{T_*}H^2}\big) \lesssim 2^{-k}\big( \| P_{\leq j}\phi\|_{X^3}+\| P_{\leq j-1}\phi\|_{X^3}\big)  \\
&\lesssim 2^{-k+j}c_{\delta,j}[\phi] .
\end{align*}
This gives \eqref{est:fe} if $j\leq k$.
When $k<j$, we use the difference bound \eqref{bd:d1} instead, to obtain
\begin{align*}
\| P_k\partial_x^2(u^j-u^{j-1})\|_{L^\infty_TL^2}&\lesssim 2^k\| \partial_x(u^j-u^{j-1})\|_{L^\infty_TL^2}\lesssim_{M_0} 2^k d^1_\delta (P_{\leq j}\phi,P_{\leq j-1}\phi) \\
&\lesssim 2^k\big( \| P_j\partial_x\phi\|_{L^2}+\| P_j\phi\|_{L^\infty}+\delta^{-\frac43}\| \phi\|_{L^\infty}\| P_j\phi\|_{L^2}\big)  \\
&\lesssim 2^k\big( 2^{-j}+2^{-\frac32j}+M_02^{-2j}\big) \| P_j\partial_x\phi\|_{H^1} \lesssim_{M_0} 2^{k-j}c_{\delta,j}[\phi],
\end{align*}
as desired.
\end{proof}


\bibliographystyle{abbrv}
\bibliography{bib_INLS}

@article{Akahori,
  title={Asymptotic behavior of dark multi-solitons to the intermediate nonlinear {S}chr{\"o}dinger equation},
  author={Akahori, Takafumi},
  journal={Partial Differential Equations in Applied Mathematics},
  pages={101273},
  year={2025},
}

@article{Burq,
  title={Nonlinear interpolation and the flow map for quasilinear equations},
  author={Alazard, Thomas and Burq, Nicolas and Ifrim, Mihaela and Tataru, Daniel and Zuily, Claude},
  journal={J. Eur. Math. Soc., to appear. arXiv preprint arXiv:2410.06909},
  year={2024}
}

@article{Badreddine1,
    AUTHOR = {Badreddine, Rana},
     TITLE = {On the global well-posedness of the {C}alogero-{S}utherland
              derivative nonlinear {S}chr\"{o}dinger equation},
   JOURNAL = {Pure Appl. Anal.},
  FJOURNAL = {Pure and Applied Analysis},
    VOLUME = {6},
      YEAR = {2024},
    NUMBER = {2},
     PAGES = {379--414},
}

@article{Badreddine2,
    AUTHOR = {Badreddine, Rana},
     TITLE = {Traveling waves and finite gap potentials for the
              {C}alogero-{S}utherland derivative nonlinear {S}chr\"{o}dinger
              equation},
   JOURNAL = {Ann. Inst. H. Poincar\'{e} C Anal. Non Lin\'{e}aire},
  FJOURNAL = {Annales de l'Institut Henri Poincar\'{e} C. Analyse Non
              Lin\'{e}aire},
    VOLUME = {42},
      YEAR = {2025},
    NUMBER = {4},
     PAGES = {1037--1092},
}

@article{Badreddine3,
    AUTHOR = {Badreddine, Rana},
     TITLE = {Zero-dispersion limit of the {C}alogero-{M}oser derivative
              {NLS} equation},
   JOURNAL = {SIAM J. Math. Anal.},
  FJOURNAL = {SIAM Journal on Mathematical Analysis},
    VOLUME = {56},
      YEAR = {2024},
    NUMBER = {6},
     PAGES = {7228--7249},
}

@article{B-deM-S,
    AUTHOR = {Barros, Vanessa and de Moura, Roger and Santos, Gleison},
     TITLE = {Local well-posedness for the nonlocal derivative nonlinear
              {S}chr\"{o}dinger equation in {B}esov spaces},
   JOURNAL = {Nonlinear Anal.},
  FJOURNAL = {Nonlinear Analysis. Theory, Methods \& Applications. An
              International Multidisciplinary Journal},
    VOLUME = {187},
      YEAR = {2019},
     PAGES = {320--338},
}

@article{Bona-Smith,
    AUTHOR = {Bona, J. L. and Smith, R.},
     TITLE = {The initial-value problem for the {K}orteweg-de {V}ries
              equation},
   JOURNAL = {Philos. Trans. Roy. Soc. London Ser. A},
  FJOURNAL = {Philosophical Transactions of the Royal Society of London.
              Series A. Mathematical and Physical Sciences},
    VOLUME = {278},
      YEAR = {1975},
    NUMBER = {1287},
     PAGES = {555--601},
}

@article{CFL,
  title={On the well-posedness of the intermediate nonlinear {S}chr{\"o}dinger equation on the line},
  author={Chapouto, Andreia and Forlano, Justin and Laurens, Thierry},
  journal={arXiv preprint arXiv:2511.00302},
  year={2025}
}

@article{Chen,
  title={The defocusing {C}alogero--{M}oser derivative nonlinear {S}chr{\"o}dinger equation with a nonvanishing condition at infinity},
  author={Chen, Xi},
  journal = {SIAM J. Math. Anal.},
  fjournal = {SIAM Journal on Mathematical Analysis},
  volume = {58},
  number = {1},
  pages = {182--205},
  year = {2026},
}

@article {colliander-kdv,
    AUTHOR = {Colliander, J. and Keel, M. and Staffilani, G. and Takaoka, H.
              and Tao, T.},
     TITLE = {Sharp global well-posedness for {K}d{V} and modified {K}d{V}
              on {$\mathbb{R}$} and {$\mathbb{T}$}},
   JOURNAL = {J. Amer. Math. Soc.},
  FJOURNAL = {Journal of the American Mathematical Society},
    VOLUME = {16},
      YEAR = {2003},
    NUMBER = {3},
     PAGES = {705--749},
}

@article{DeMoura,
    AUTHOR = {de Moura, Roger Peres},
     TITLE = {Well-posedness for the nonlocal nonlinear {S}chr\"{o}dinger
              equation},
   JOURNAL = {J. Math. Anal. Appl.},
  FJOURNAL = {Journal of Mathematical Analysis and Applications},
    VOLUME = {326},
      YEAR = {2007},
    NUMBER = {2},
     PAGES = {1254--1267},
}

@article{Frank-Read,
  title={Jost solutions and direct scattering for the continuum {C}alogero-{M}oser equation},
  author={Frank, Rupert L. and Read, Larry},
  journal={arXiv preprint arXiv:2510.11403},
  year={2025}
}

@inproceedings{Gerard,
    AUTHOR = {G\'{e}rard, P.},
     TITLE = {The {C}auchy problem for the {G}ross-{P}itaevskii equation},
   JOURNAL = {Ann. Inst. H. Poincar\'{e} C Anal. Non Lin\'{e}aire},
  FJOURNAL = {Annales de l'Institut Henri Poincar\'{e} C. Analyse Non
              Lin\'{e}aire},
    VOLUME = {23},
      YEAR = {2006},
    NUMBER = {5},
     PAGES = {765--779},
}

@article{GL,
    AUTHOR = {G\'{e}rard, Patrick and Lenzmann, Enno},
     TITLE = {The {C}alogero-{M}oser derivative nonlinear {S}chr\"{o}dinger
              equation},
   JOURNAL = {Comm. Pure Appl. Math.},
  FJOURNAL = {Communications on Pure and Applied Mathematics},
    VOLUME = {77},
      YEAR = {2024},
    NUMBER = {10},
     PAGES = {4008--4062},
}

@article {HGKV,
    AUTHOR = {Harrop-Griffiths, Benjamin and Killip, Rowan and Vi\c{s}an,
              Monica},
     TITLE = {Sharp well-posedness for the cubic {NLS} and m{K}d{V} in
              {$H^s(\mathbb{R})$}},
   JOURNAL = {Forum Math. Pi},
  FJOURNAL = {Forum of Mathematics. Pi},
    VOLUME = {12},
      YEAR = {2024},
     PAGES = {Paper No. e6, 86},
}

@article{Hogan-Kowalski,
    AUTHOR = {Hogan, James and Kowalski, Matthew},
     TITLE = {Turbulent threshold for continuum {C}alogero-{M}oser models},
   JOURNAL = {Pure Appl. Anal.},
  FJOURNAL = {Pure and Applied Analysis},
    VOLUME = {6},
      YEAR = {2024},
    NUMBER = {4},
     PAGES = {941--954},
}

@article{Ifrim-Tataru,
    AUTHOR = {Ifrim, Mihaela and Tataru, Daniel},
     TITLE = {Local well-posedness for quasi-linear problems: a primer},
   JOURNAL = {Bull. Amer. Math. Soc. (N.S.)},
  FJOURNAL = {American Mathematical Society. Bulletin. New Series},
    VOLUME = {60},
      YEAR = {2023},
    NUMBER = {2},
     PAGES = {167--194},
}

@article{jJeong-Kim,
  title={Quantized blow-up dynamics for {C}alogero--{M}oser derivative nonlinear {S}chr{\"o}dinger equation},
  author={Jeong, Uihyeon and Kim, Taegyu},
  journal={arXiv preprint arXiv:2412.12518},
  year={2024}
}

@article{killip,
    AUTHOR = {Killip, Rowan and Laurens, Thierry and Vi\c{s}an, Monica},
     TITLE = {Scaling-critical well-posedness for continuum
              {C}alogero-{M}oser models on the line},
   JOURNAL = {Commun. Am. Math. Soc.},
  FJOURNAL = {Communications of the American Mathematical Society},
    VOLUME = {5},
      YEAR = {2025},
     PAGES = {284--320},
}

@article{kim-kim-kwon,
  title={Construction of smooth chiral finite-time blow-up solutions to {C}alogero--{M}oser derivative nonlinear {S}chr{\"o}dinger equation},
  author={Kim, Kihyun and Kim, Taegyu and Kwon, Soonsik},
  journal={arXiv preprint arXiv:2404.09603},
  year={2024}
}

@article{kim-Kwon,
  title={Soliton resolution for {C}alogero--{M}oser derivative nonlinear {S}chr{\"o}dinger equation},
  author={Kim, Taegyu and Kwon, Soonsik},
  journal={arXiv preprint arXiv:2408.12843},
  year={2024}
}

@article{Matsuno1,
    AUTHOR = {Matsuno, Yoshimasa},
     TITLE = {{$N$}-soliton formulae for the intermediate nonlinear
              {S}chr\"{o}dinger equation},
   JOURNAL = {Inverse Problems},
  FJOURNAL = {Inverse Problems. An International Journal on the Theory and
              Practice of Inverse Problems, Inverse Methods and Computerized
              Inversion of Data},
    VOLUME = {17},
      YEAR = {2001},
    NUMBER = {3},
     PAGES = {501--514},
}

@article{Matsuno2,
    AUTHOR = {Matsuno, Yoshimasa},
     TITLE = {Multiperiodic and multisoliton solutions of a nonlocal
              nonlinear {S}chr\"{o}dinger equation for envelope waves},
   JOURNAL = {Phys. Lett. A},
  FJOURNAL = {Physics Letters. A},
    VOLUME = {278},
      YEAR = {2000},
    NUMBER = {1-2},
     PAGES = {53--58},
}

@article{Matsuno3,
    AUTHOR = {Matsuno, Yoshimasa},
     TITLE = {Asymptotic solutions of the nonlocal nonlinear
              {S}chr\"{o}dinger equation in the limit of small dispersion},
   JOURNAL = {Phys. Lett. A},
  FJOURNAL = {Physics Letters. A},
    VOLUME = {309},
      YEAR = {2003},
    NUMBER = {1-2},
     PAGES = {83--89},
}

@article{Matsuno4,
    AUTHOR = {Matsuno, Yoshimasa},
     TITLE = {Linear stability of multiple dark solitary wave solutions of a
              nonlocal nonlinear {S}chr\"{o}dinger equation for envelope
              waves},
   JOURNAL = {Phys. Lett. A},
  FJOURNAL = {Physics Letters. A},
    VOLUME = {285},
      YEAR = {2001},
    NUMBER = {5-6},
     PAGES = {286--292},
}

@article{Matsuno5,
    AUTHOR = {Matsuno, Yoshimasa},
     TITLE = {Calogero-{M}oser-{S}utherland dynamical systems associated
              with nonlocal nonlinear {S}chr\"{o}dinger equation for
              envelope waves},
   JOURNAL = {J. Phys. Soc. Japan},
  FJOURNAL = {Journal of the Physical Society of Japan},
    VOLUME = {71},
      YEAR = {2002},
    NUMBER = {6},
     PAGES = {1415--1418},
}

@article{Matsuno7,
    AUTHOR = {Matsuno, Yoshimasa},
     TITLE = {A {C}auchy problem for the nonlocal nonlinear
              {S}chr\"{o}dinger equation},
   JOURNAL = {Inverse Problems},
  FJOURNAL = {Inverse Problems. An International Journal on the Theory and
              Practice of Inverse Problems, Inverse Methods and Computerized
              Inversion of Data},
    VOLUME = {20},
      YEAR = {2004},
    NUMBER = {2},
     PAGES = {437--445},
}

@article{Matsuno9,
    AUTHOR = {Matsuno, Yoshimasa},
     TITLE = {Multiphase solutions and their reductions for a nonlocal
              nonlinear {S}chr\"{o}dinger equation with focusing
              nonlinearity},
   JOURNAL = {Stud. Appl. Math.},
  FJOURNAL = {Studies in Applied Mathematics},
    VOLUME = {151},
      YEAR = {2023},
    NUMBER = {3},
     PAGES = {883--922},
}

@article{P,
  title={Intermediate nonlinear {S}chr{\"o}dinger equation for internal waves in a fluid of finite depth},
  author={Pelinovsky, Dmitry},
   JOURNAL = {Phys. Lett. A},
  FJOURNAL = {Physics Letters. A},
  volume={197},
  number={5-6},
  pages={401--406},
  year={1995},
}

@article{PG,
    AUTHOR = {Pelinovsky, Dmitry E. and Grimshaw, Roger H. J.},
     TITLE = {A spectral transform for the intermediate nonlinear
              {S}chr\"{o}dinger equation},
   JOURNAL = {J. Math. Phys.},
  FJOURNAL = {Journal of Mathematical Physics},
    VOLUME = {36},
      YEAR = {1995},
    NUMBER = {8},
     PAGES = {4203--4219},
}

@article{shatah,
    AUTHOR = {Shatah, Jalal},
     TITLE = {Normal forms and quadratic nonlinear {K}lein-{G}ordon
              equations},
   JOURNAL = {Comm. Pure Appl. Math.},
  FJOURNAL = {Communications on Pure and Applied Mathematics},
    VOLUME = {38},
      YEAR = {1985},
    NUMBER = {5},
     PAGES = {685--696},
}

@article {Tao-2001,
    AUTHOR = {Tao, Terence},
     TITLE = {Global regularity of wave maps. {II}. {S}mall energy in two
              dimensions},
   JOURNAL = {Comm. Math. Phys.},
  FJOURNAL = {Communications in Mathematical Physics},
    VOLUME = {224},
      YEAR = {2001},
    NUMBER = {2},
     PAGES = {443--544},
}

\end{document}